\documentclass[psamsfonts]{amsart}
\usepackage{amssymb,amsfonts}
\usepackage[all,arc]{xy}
\usepackage{enumerate}
\usepackage[title]{appendix}
\usepackage{booktabs}
\usepackage{verbatim}
\usepackage{microtype}
\usepackage{mathrsfs}
\usepackage{tikz}
\usetikzlibrary{arrows,snakes,backgrounds}
\usetikzlibrary{decorations.pathreplacing,calligraphy}
\usepackage{color}   
\usepackage{marginnote}
\usepackage{tikz-cd}
\usetikzlibrary{positioning}

\usepackage{enumitem}
\usepackage{hyperref}
\usepackage{MnSymbol} 
\hypersetup{
    colorlinks=true, 
    linktoc=all,     
    citecolor=black,
    filecolor=black,
    linkcolor=blue,
    urlcolor=black
}
\usepackage{cleveref}
\usetikzlibrary{arrows}

\newtheorem{thm}{Theorem}[section]
\newtheorem{cor}[thm]{Corollary}
\newtheorem{prop}[thm]{Proposition}
\newtheorem{lem}[thm]{Lemma}

\newtheorem{claim}[thm]{Claim}

\newtheorem*{openproblem*}{Problem}
\newtheorem*{quest*}{Question}
\newtheorem*{problem*}{Problem}

\theoremstyle{definition}
\newtheorem{defn}[thm]{Definition}

\newtheorem{con}[thm]{Construction}
\newtheorem{exmp}[thm]{Example}

\newtheorem*{cond*}{Condition}

\theoremstyle{remark}
\newtheorem{rem}[thm]{Remark}

\newcommand{\bR}{\mathbb{R}}

\newcommand{\bZ}{\mathbb{Z}}
\newcommand{\floor}[1]{\lfloor #1 \rfloor}

\newcommand\Diff{\mathrm{Diff}}

\newcommand\Cont{\mathrm{Cont}}

\newcommand\BDiff{\mathrm{BDiff}}

\newcommand\BCont{\mathrm{BCont}}
\newcommand\dDiff{\mathrm{Diff}^{\delta}}

\newcommand\dCont{\mathrm{Cont}^{\delta}}

\newcommand\BdDiff{\mathrm{BDiff}^{\delta}}

\newcommand\BdCont{\mathrm{BCont}^{\delta}}

\addtocontents{toc}{\protect\setcounter{tocdepth}{1}}
\makeatletter
\let\c@equation\c@thm
\makeatother
\numberwithin{equation}{section}
\title[Thurston's fragmentation]{Thurston's fragmentation and c-principles}

\author{Sam Nariman}
\email{snariman@purdue.edu}
\address{Department of Mathematics\\
  Purdue University\\
150 N. University Street\\
West Lafayette, IN 47907-2067\\
}
\dedicatory{To the memory of John Mather}

\begin{document}

\begin{abstract}
In this paper, we generalize the original idea of Thurston for the so-called Mather-Thurston's theorem for foliated bundles to prove new variants of this theorem for PL homeomorphisms and contactormorphisms. These versions answer questions posed by Gelfand -Fuks (\cite[Section 5]{MR0339195}) and Greenberg (\cite{MR1200422}) on PL foliations and Rybicki  (\cite[Section 11]{MR2729009}) on contactomorphisms. The interesting point about the original Thurston's technique compared to the better-known Segal-McDuff's proof of the Mather-Thurston theorem is that it gives a {\it compactly supported} c-principle theorem without knowing the relevant local statement on open balls. In the appendix, we show that Thurston's fragmentation implies the non-abelian Poincare duality theorem and its generalization using blob complexes (\cite[Theorem 7.3.1]{MR2978449}).
\end{abstract}
\maketitle

\section{Introduction}
Thurston (\cite{thurston1974foliations}) found a remarkable relation between the identity component of diffeomorphism groups of an $n$-dimensional compact manifold $M$ and ``singular" foliations, induced by Haefliger structures (see \cite{haefliger1971homotopy}), on $M$. His theorem can be thought of as a homology h-principle theorem or a c-principle theorem (see \cite[Theorem 3.3]{fuks1974quillenization}) between the space of genuine foliations on $M$-bundles that are transverse to the fiber $M$ and the space of singular foliations on $M$-bundles whose normal bundles are isomorphic to the vertical tangent bundle. 

More concretely, a Haefliger structure $\mathcal{H}$ comes with the data of a vector bundle $\nu\mathcal{H}$ (see \cite{haefliger1971homotopy}) which is called the normal bundle of the Haefliger structure $\mathcal{H}$ and a germ of a foliation near the zero section of $\nu\mathcal{H}$ that is transverse to the fibers but not necessarily to the zero section. So the intersection of this germ of foliation with the zero section will be described by a ``singular" foliation or a Haefliger structure. 

For a closed manifold $N$ we consider a Haefliger structure on a product bundle $\pi\colon N\times M\to N$ whose normal bundle is isomorphic to the vertical tangent bundle of $\pi$. One can ask whether this Haefliger structure is homotopic to a genuine foliation on $N\times M$ that is transverse to the fibers. In other words, whether an h-principle theorem holds for this formal data. This is not in general true, however, Thurston's theorem in (\cite{thurston1974foliations}) implies that there exists a ``cobordism" (hence a c-principle) of a trivial $M$-bundle with a Haefliger structure whose normal bundle is isomorphic to the vertical tangent bundle that starts from the bundle $\pi$ and ends with a bundle $N'\times M\to N'$ with a genuine foliation on the total space $N'\times M$ which is transverse to the fibers. 

 The space of foliation on a trivial $M$-bundle transverse to fibers is related to  $\Diff_0(M)$, the identity component of the diffeomorphism group, and the formal space that does not have the transversality condition is related to a section space over $M$ whose fiber is at least $n$-connected. Thurston showed that these two spaces satisfy a certain {\it fragmentation property}.  It is easier to state this property for $\Diff_0(M)$. So let $\{ U_i\}_{i}$ be a finite open cover for $M$. Fragmentation with respect to this cover means that any element $f\in \Diff_0(M)$ can be written as a composition of diffeomorphisms $f_j$ such that  $f_j$ is compactly supported in some element of the cover  $\{ U_i\}_{i}$. He used the fragmentation property to filter these two spaces and compare their filtration quotients to prove his c-principle theorem. In this paper, we first improve and make the method of Thurston more abstract to be able to apply it to other geometric structures.

To set up a more general context, let $F: (\mathsf{Mfld}^{\partial}_n)^{op}\to \mathsf{S}$ be a presheaf from the category of smooth $n$-manifolds (possibly with nonempty boundary) with smooth embeddings as morphisms to a convenient category of spaces $\mathsf{S}$. For our purpose, we shall consider the category of simplicial sets or compactly generated Hausdorff spaces. Let $F^f$ be the homotopy sheafification of $F$ with respect to $1$-good covers meaning contractible open sets whose nontrivial intersections are also contractible (see \cite{MR3138384} for more details). One can describe the value of $F^f(M)$ as the space of sections of the bundle $\text{Fr}(M)\times _{\mathrm{GL}_n(\bR)} F(\bR^n)\to M$, where $\text{Fr}(M)$ is the frame bundle of $M$.  We say $F$ satisfies an h-principle if  the natural map from the functor to its homotopy sheafification 
\[
j:F(M)\to F^f(M),
\]
 induces a weak equivalence and we say it satisfies the c-principle if the above map is a homology isomorphism. 
 
 Some important examples of such presheaf in the manifold topology are the space of generalized Morse functions (\cite{MR744854}), the space of framed functions (\cite{MR882699}), the space of smooth functions on $M^n$ that avoid singularities of codimension $n+2$ (this is, in general, a c-principle theorem, see \cite{MR1168473}), the space of configuration of points with labels in a connected space (\cite{MR0358766}), etc. h- and c-principle theorems (see \cite{MR1909245}) come in many different forms but the general philosophy is that a space of a geometric significance $F(M)$ which is sometimes called ``holonomic solutions" is homotopy equivalent or homology isomorphic to ``formal solutions" $F^f(M)$ (the superscript $f$ stands for formal). The space of formal solutions, $F^f(M)$,  is more amenable to homotopy theory since it is often the space of sections of a fiber bundle and therefore, it is easy to check its homotopy sheaf property with respect to certain covers. Hence, from the homotopy theory point of view, proving the h-principle theorem consists of a ``local statement" which is an equivalence of holonomic solutions and formal solutions on open balls, and a ``local to the global statement"  which is a homotopy sheaf property for the geometric functor of holonomic solutions. 

Thurston in (\cite{thurston1974foliations}), however, found a remarkable {\it compactly supported} c-principle theorem without knowing the ``local statement". The main goal of this paper is to abstract his ideas to prove new variants of compactly supported c-principle theorems without knowing the local statement. To briefly explain his compactly supported c-principle theorem, let $\text{Fol}_c(M):=\overline{\BDiff_c(M)}$\footnote{Historically as in \cite{thurston1974foliations}, for any topological group $G$, the space $\overline{\mathrm{B}G}$ is defined to be the homotopy fiber of the map $\mathrm{B}G^{\delta}\to \mathrm{B}G$ where $G^{\delta}$ is the group $G$ with discrete topology and the map induced by the identity homomorphism $G^{\delta}\to G$.} be the realization of the semisimplicial set whose $k$-simplices are given by the set of codimension $n$ foliations on $M^n \times \Delta^k$ that are transverse to the fibers of the projection $M \times \Delta^k\to \Delta^k$ and the foliations are horizontal outside of some compact set. 

To describe the space of formal solutions in this case, we need to recall the notion of Haefliger classifying space which is the space of formal solutions on an open ball. Let $\text{Fol}^f(\bR^n):=\overline{\mathrm{B}\Gamma_n}$ \footnote{Historically as in \cite{haefliger1971homotopy}, $\overline{\mathrm{B}\Gamma_n}$ is defined to be the homotopy fiber of the map $\nu\colon \mathrm{B}\Gamma_n\to \mathrm{B}\text{GL}_n(\bR)$ where  $ \mathrm{B}\Gamma_n$ is the classifying space of codimension $n$ Haefliger structures and $\nu$ classifies their normal bundles. } be the realization of a semisimplicial set whose $k$-simplices are given by the set of the {\it germs}  of codimension $n$ foliations  on $\bR^n \times \Delta^k$ around $\{0\} \times \Delta^k$ that are transverse to the fibers of the projection $\bR^n \times \Delta^k\to \Delta^k$. After fixing a base section of the space of sections of $\text{Fr}(M)\times _{\mathrm{GL}_n(\bR)} \text{Fol}^f(\bR^n)\to M$, we can define the support of sections to be the set on which they take different values from the base section. Let $\text{Fol}^f_c(M)$  be the space of compactly supported sections with respect to the fixed base section. Thurston proved that there exists a natural map $\text{Fol}_c(M)\to \text{Fol}^f_c(M)$ which induces a homology isomorphism. 

Although Segal later proved  (see \cite{segal1978classifying})   the local statement that $\overline{\BDiff(\bR^n)}$ is homology isomorphic to $\overline{\mathrm{B}\Gamma_n}$ which led to a different proof (\cite{mcduff1979foliations}) of Thurston's theorem, Thurston's original proof of the fact that a natural map $\text{Fol}_c(M)\to \text{Fol}^f_c(M)$ induces a homology isomorphism did not use this local statement.

The main idea is, given a metric on $M$ satisfying a mild condition (see \Cref{complete}), Thurston gives a compatible filtration on the space of foliated $M$-bundles  $\text{Fol}_c(M)$ and the space of formal solutions $\text{Fol}_c^f(M)$ which is a section space, and compares the spectral sequences of these filtrations to prove his compactly supported c-principle theorem. These filtrations are inspired by his idea of ``fragmenting" diffeomorphisms of manifolds that are isotopic to the identity.

\subsection{c-principle theorems via fragmentation}
Part of the method, Thurston used to prove his c-principle theorem is of course specific to foliation theory. In particular, the fact is that the local statement, in that case, was very nontrivial and the way he proved the compactly supported version without the local statement is specific to foliation theory. However, we show that given the local statement (which is often the easy case unlike foliation theory), we can still apply Thurston's method to obtain a compactly supported c-principle theorem. Then we also use this general strategy to prove versions of Thurston's theorem for other geometric structures that were conjectured to hold. 

Normally in c-principle theorems, the local statement is that the map $ F(\bR^n)\to  F^f(\bR^n)$ is a homology isomorphism or even a homotopy equivalence. In this context, when the functor is defined on manifolds with boundaries, we would like to consider closed disks instead. To do this, we first need to define $ F^f(-)$ on $\mathsf{Mfld}^{\partial}_n$ and in particular on closed disks. Fixing the space $F^f(\bR^n)$, we can define $F^f(-)$ on other manifolds ``linearly" as follow.
\begin{defn}\label{formal} Given that the group $\mathrm{GL}_n(\bR)$ acts both on $ F^f(\bR^n)$ since it acts on $\bR^n$ and it also acts on the frame bundle $\text{Fr}(M)$, we can form the following natural bundle over $M$ 
 \[
\text{Fr}(M)\times _{\mathrm{GL}_n(\bR)} F^f(\bR^n)\to M,
 \]
whose fiber is $ F^f(\bR^n)$. Let the space of formal solutions, $F^f(M)$, be the space of sections of this fiber bundle.
 \end{defn}
 \begin{rem} Note that since $D^n$ is contractible $F^f(D^n)\simeq F^f(\bR^n)$ and in all the examples of c-principle in the introduction, the common feature is the space $ F^f(D^n)$ is in fact at least $(n-1)$-connected.  Therefore, the cosheaf of compactly supported sections $F_c^f(-)$ satisfies the fragmentation property and non-abelian Poincare duality. 
\end{rem}
 The set-up of the c-principle theorem that we are interested in is the following: we have a natural transformation $\iota: F(-)\to F^f(-)$ that respects the choice of base sections. Hence, for any manifold $M$, we have an induced map 
 \[
 F_c(M)\to F_c^f(M).
 \]
 We want to find conditions under which the above map induces a homology isomorphism. 
 \begin{defn}\label{geod}For a given metric space $(M,d)$, the intrinsic metric between two points $x$ and $y$ in $M$ is defined to be the infimum of the lengths of all paths from $x$ to $y$. If the intrinsic metric agrees with the original metric $d$ on $M$, we call $(M,d)$ a length metric space. And if additionally there always exists a path that achieves the infimum of length (a geodesic) between all pairs of points, we call $(M,d)$ a geodesic space.
\end{defn}

\begin{defn}\label{complete}
Let  $s_0\in F(M)$ be a fixed global section and we fix a metric $d(-,-)$ on $M$. We suppose that the metric is complete and $(M,d)$ is a geodesic space and there exists an $\epsilon>0$  such that all balls of radius $\epsilon$ are geodesically convex. If $M$ is compact, then these two conditions are automatically given. For any other element $s\in F(M)$, we define the notion of {\it support}, $\text{supp}(s)$, with respect to $s_0$ to be the closure of points in $M$ at which the stalk of $s$ and $s_0$ are different. Now let $F_{\epsilon}(M, s_0)$ denote the subspace of $F(M)$ consisting of elements $s$ such that the support of $s$ can be covered by $k$ geodesically convex balls of radius $2^{-k}\epsilon$ for some positive integer $k$. Also, we can define the subspace of compactly supported elements. We shall suppress the fixed global section $s_0$ from the notation for brevity.  In the case of a non-empty boundary, we assume that the supports of all elements of $F_c(M, s_0)$ and $F_c^f(M, s_0)$ are away from the boundary.
\end{defn}
\begin{defn}\label{prop}
We say the functor $F$ satisfies the fragmentation property if the inclusion $F_{\epsilon}(M)\to F_c(M)$ is a weak equivalence for all $M$. 
\end{defn}
\begin{defn}\label{condi}
We say $F: (\mathsf{Mfld}^{\partial}_n)^{op}\to \mathsf{S}$ is {\it good}, if it satisfies 
\begin{itemize}
\item The subspace of elements with empty support in $F(M)$ is contractible. 

\item For an open subset $U$ of a manifold $M$, the inclusion $F_c(U)\to F_c(M)$ is an open embedding (We will consider the weaker condition in \Cref{condi}).
\item Let $U$ and $V$ be open disks.  All embeddings $U\hookrightarrow V$ induces a homology isomorphism between $F_c(U)$ and  $F_c(V)$.
\item For each finite family of open sets $U_1,\dots U_k$ such that they  are pairwise disjoint and are contained in an open set $U_0$, we have a permutation invariant map
\[
\mu^{U_0}_{U_1,\dots, U_k}\colon \prod_{i=1}^k F_c(U_i)\to F_c(U_0),
\]
where this map satisfies the obvious associativity conditions and for $U_0=\bigcup_{i=1}^{k} U_i$, the map $\mu^{U_0}_{U_1,\dots, U_k}$ is a weak equivalence.
\item Let $\partial_1$ be the northern-hemisphere boundary of $D^n$.  Let $F(D^n, \partial_1)$ be the subspace of $F(D^n)$ that restricts to the base element in a germ of $\partial_1$ inside $D^n$. We assume $F(D^n, \partial_1)$ is acyclic.  
\end{itemize}
  \end{defn}

\begin{thm}\label{main}
Let $F$ be a good sheaf on $n$-dimensional manifolds such that
\begin{itemize}
\item $F^f(D^n)$ is at least $(n-1)$-connected. \item $F$ has the fragmentation property. 
\end{itemize}
Then for any manifold $M$ which admits a metric that makes $M$ a complete geodesic space, then the map
 \[
 F_c(M)\to F_c^f(M),
 \]
induces a homology isomorphism.
 \end{thm}
\begin{rem}
 The connectivity hypothesis in \Cref{fragment} and \Cref{main} is improved by one compared to the original Thurston's deformation technique. And as we shall see, this improvement will be useful to prove Mather-Thurston-type theorems for different geometric structures. One can also use this method to give a different proof of McDuff's theorem on the local homology of volume-preserving diffeomorphisms (\cite{mcduff1983local, MR707329}) using the methods of this paper. In that case $F(D^n)$ is at best $(n-1)$-connected (see \cite[Remark 2, part (a)]{haefliger1971homotopy}).
\end{rem}
It would be interesting to see if Thurston's method gives a different proof of Vassiliev's c-principle theorem (\cite{MR1168473}). In the last section, we discuss how one could use Thurston's fragmentation idea for the space of functions on $M$ avoiding singularities of codimension $\text{dim}(M)+2$. However, our main motivation still lies in foliation theory. 

Fragmentation property of foliation with different transverse structures (\cite{MR2509723, MR2284795, MR2729009}) have been extensively studied and conjecturally it is expected that an analog of Thurston's theorem or so-called Mather-Thurston's theory (for PL-homeomorphisms see \cite[Section 5]{MR0339195} and for a different version see also \cite{MR1200422}, for contactomorphisms see \cite{MR2729009}) should also hold for them. We prove in \Cref{newMT}  new variants of Mather-Thurston's theorem for contactomorphisms and PL-homeomorphisms which were conjectured by  Rybicki and Gelfand-Fuks/Greenberg respectively.

Recently, there were new geometric approaches to Mather-Thurston's theory due to Meigniez (\cite{meigniez2018quasi}) and Freedman (\cite{freedman2020controlled}). However, in this paper, we follow Mather's account (\cite{mather2011homology}) of Thurston's proof of this remarkable theorem in foliation theory. McDuff followed in \cite{mcduff1980homology, mcduff1979foliations} Segal's method (\cite{segal1978classifying}) to find a different proof of Mather-Thurston's theorem and she proved the same theorem for the volume-preserving case (\cite{MR707329, MR699012, mcduff1983local}). The techniques in Segal and McDuff's approach and in particular, their group completion theorem (\cite{mcduff1976homology}) are now well-understood tools in homotopy theory.  The author hopes that this paper also makes Thurston's ideas available to a broader context. 

\subsection{Mather-Thurston theory for new transverse structures} We consider two different transverse structures of foliated bundles for which the fragmentation properties were known and hence conjecturally the analogs of Mather-Thurston's theorem were posed (\cite{MR2729009, MR1200422}). We shall first recall these transverse structures.

\begin{defn}\label{foliatedbdle}
\begin{itemize}
\item Let $M$ be a smooth odd-dimensional manifold with a fixed contact structure $\alpha$. Let $\text{Fol}_c(M, \alpha)$ be the realization of the simplicial set whose $k$-simplices are given by the set of codimension $\text{dim}(M)$ foliations  on $M \times \Delta^k$ that are transverse to the fibers of the projection $M \times \Delta^k\to \Delta^k$ and the holonomies are compactly supported contactomorphisms of the fiber $M$.
\item Let $M$ be a PL $n$-dimensional manifold. Let $\text{Fol}^{\text{PL}}_c(M)$ be the realization of the simplicial set whose $k$-simplices are given by the set of codimension $\text{dim}(M)$ foliations  on $M \times \Delta^k$ that are transverse to the fibers of the projection $M \times \Delta^k\to \Delta^k$ and the holonomies are compactly supported PL-homeomorphisms of the fiber $M$.
\end{itemize}
\end{defn}
Analogue of Mather-Thurston's theorem in these cases can be summarized as follows.
\begin{thm}\label{MT}
The functors $\textnormal{\text{Fol}}_c(M, \alpha)$ and $\textnormal{\text{Fol}}^{\textnormal{\text{PL}}}_c(M)$ satisfy the c-principle.
\end{thm}
\begin{rem}
Gael Meigniez told the author that he has a forthcoming paper to show that the PL case could also be obtained using his geometric proof for the smooth case (\cite{meigniez2018quasi}) and there is a work in progress to use his method in the transverse contact structure. 
\end{rem}

\subsubsection{Perfectness and Mather-Thurston's theorems} Often in h- and c-principles theorems, the {\it formal solutions} is easier to study than the {\it holonomic solutions}. However, Thurston used the Mather-Thurston theorem and the perfectness of the identity component of smooth diffeomorphism groups to improve the connectivity of the Haefliger space which is on the formal side of the theorem. Similarly, our c-principle theorems and the perfectness results  in \cite{MR2729009, MR2509723, MR2284795} can be used to improve the connectivity results of the corresponding Haefliger structures. In particular, as a corollary (see \Cref{improve}) for transverse contact structures and flag of foliations, we obtain the following.
\begin{cor}
The Haefliger classifying space $\overline{\mathrm{B}\Gamma_{2n+1,ct}}$ of codimension $2n+1$ Haefliger structures with a transverse contact structure is  at least $(2n+2)$-connected. 
\end{cor}
These connectivity ranges are improved by one from the previously known ranges (see \cite[Proposition 7.4]{mcduff1987applications}).

However, for a PL manifold $M$, unlike other transverse structures, curiously the perfectness result is not known in general.  It was asked by Epstein (\cite{epstein1970simplicity}) whether $\text{PL}_0(M)$ as an abstract group is perfect and he proved it for $\text{PL}_0(S^1)$. In \cite{nariman2022flat}, the author used the c-principle for $\textnormal{\text{Fol}}^{\textnormal{\text{PL}}}_c(M)$ and the work of Greenberg (\cite{MR1200422}) to show that $\text{PL}_0(M)$ is perfect for any closed surface $M$. 

\subsection{Organization:}  In \Cref{sec3}, we discuss fragmentation homotopy and we  improve it to prove \Cref{fragment}. In \Cref{sec4}, we apply Thurston's fragmentation ideas in foliation theory in a broader context to prove \Cref{main}. In \Cref{newMT}, we prove a compactly supported version of Mather-Thurston's theorem for PL and contact transverse structures. In these cases still, the local statements are not known and therefore, the non-compactly supported versions are still open, in \Cref{sec2}, we use microfibration techniques to show that Thurston's fragmentation method implies the non-abelian Poincar{\' e} duality.
 \subsection*{Acknowledgment} I would like to thank John Mather for his encouragement and correspondence in 2016 about his paper \cite{mather2011homology}. Mather described his notes on Thurston's lectures  in an email: ``{\it The proof that I wrote up was based on 
lectures that Thurston gave at Harvard.  The lectures were sketchy and it 
was really hard to write up the proof.  I spent 14 months on it "}. We are grateful for his detailed account of Thurston's intuition which was the main inspiration for this paper.  I would like to thank Mike Freedman for many discussions around Mather-Thurston's theorem and for mentioning the relation to blob homology.  I would like to thank Gael Meigniez, F. Laudenbach, Y. Eliashberg, Sander Kupers, and T. Tsuboi  for their comments and discussions. The author is partially supported by NSF grant DMS-1810644,  NSF CAREER Grant DMS-2239106, and Simons Collaboration Grant award 855209. I also thank the referees for their careful reading and comments that improved the exposition and the readability of the paper.
\section{Thurston's fragmentation} \label{sec3}
In this section, we explain Thurston's idea of fragmentation and we improve the hypothesis of the connectivity of the fiber in Mather's note \cite[First deformation lemma]{mather2011homology} by one. And throughout the paper, we assume that $M$ satisfies the hypothesis in \Cref{complete}.

To explain his fragmentation idea, it is easier to start with fragmenting the space of sections. Let $\pi: E\to M$ be a Serre fibration over the manifold $M$ and we suppose $E$ is Hausdorff. Let $s_0$ be a base section of this fiber bundle. 
\begin{cond*}\label{goodsection}We assume that the base section satisfies the following homotopical property: there is a fiber preserving homotopy $h_t$ of $E$ such that $h_0=\text{id}$ and $h_1^{-1}(s_0(M))$ is a neighborhood of $s_0(M)$ in $E$ and $h_t(s_0(M))=s_0(M)$ for all $t$, in other words, the base section is a {\it good} base point in the space of sections.  We fix a metric on $M$ and assume that it is a geodesic space (see \Cref{geod}) that there exists a positive $\epsilon$ so that every ball of radius $\epsilon$ is geodesically convex.
\end{cond*}
By the support of a section $s$, we mean the closure of the points on which $s$ differs from the base section $s_0$. Let $\text{Sect}_c(\pi)$ be the space of compactly supported sections of the fiber bundle $\pi:E\to M$ equipped with the compact-open topology. 

Let $\text{Sect}_{\epsilon}(\pi)$ denote the subspace of sections $s$ such that the support of $s$ can be covered by $k$ geodesically convex balls of radius $2^{-k}\epsilon$ for some positive integer $k$. Note that there is a filtration on  $\text{Sect}_{\epsilon}(\pi)$ by the number of balls that cover the support. 

The reason for the choice of $2^{-k}\epsilon$, as we shall see in detail in \Cref{ssSpace}, is to have nice filtration quotients where the filtration is induced by the number of balls that cover the support of a section. For example, suppose the support of a section can be covered by two balls of radius $2^{-2}\epsilon$ but it cannot be covered by one ball of radius $2^{-1}\epsilon$ so it is a nontrivial element in the second term of the filtration quotients. Then one could choose those two balls to be disjoint and this phenomenon will be useful to describe the filtration quotients in particular in proving \Cref{j_1}.
\begin{thm}[Fragmentation property]\label{fragment} If the fiber of $\pi$ is at least $(n-1)$-connected, the inclusion 
\[
\text{\textnormal{Sect}}_{\epsilon}(\pi)\hookrightarrow \text{\textnormal{Sect}}_c(\pi)
\]
is a weak homotopy equivalence. 
\end{thm}
\begin{rem}
 Mather refers to the above statement as a deformation lemma in \cite{mather2011homology} and he assumed that the fiber is $n$-connected but we show that $(n-1)$-connectedness is enough. 
\end{rem}
\begin{rem}\label{Federico Cantero Moran}
In general, if the fiber of $\pi$ is $(n-k)$-connected, the same techniques apply to localize the support of the sections. For example,  for a fixed neighborhood $U$ of the $(k-1)$-skeleton, one could show  that the space $\text{\textnormal{Sect}}_c(\pi)$ is weakly equivalent to the subspace of sections that are supported in $U$ union  $s$ balls of radius $2^{-s}\epsilon$ for some non-negative integer $s$.  But this is not the direction, we want to pursue in this paper.
\end{rem}
As we shall see in \Cref{thurston}, given a $D^k$-family of sections in $ \text{\textnormal{Sect}}_c(\pi)$, we subdivide the parameter space $D^k$ and change the family up to homotopy such that on each part of this subdivision, the new family is supported in the union of $k$ balls of radius $2^{-k}\epsilon$. 

\subsection{Fragmentation homotopy}\label{thurston} Let $\{\mu_i\}_{i=1}^N$ be a partition of unity with respect to an open cover of $M$. We define a {\it fragmentation homotopy} with respect to this partition of unity. Let $\nu_0=0$ and for $j>0$, let $\nu_j$ be the function
\[
\nu_j(x)=\sum_{k=1}^j\mu_k(x).
\]
We shall write $\Delta^q$ for the standard $q$-simplex parametrized by $$\{ {\bf t}=(t_1,t_2,\dots,t_q);  0\leq t_1\leq \dots\leq t_q\leq 1\}.$$  We now consider the following map
\[
H_1:M\times \Delta^q\to M\times \Delta^q,
\]
\[
H_1(x, (t_1,t_2,\dots,t_q))=(x,(u_1,u_2,\dots,u_q)),
\]
\[
u_i(x, {\bf t})=\nu_{\floor{Nt_i}}(x)+\mu_{\floor{Nt_i}+1}(x)(Nt_i-\floor{Nt_i}).
\]
Note that $u_i$ only depends on $t_i$ and $x$. Since $H_1({\bf t},x)$ preserves the $x$ coordinate, we can define a straight line homotopy $H_t:M\times \Delta^q\to M\times \Delta^q$ from the identity to $H_1$. 
\begin{figure}[h]\label{frag}
\[
\begin{tikzpicture}
  \def\rectanglepath{-- ++(3cm,0cm)  -- ++(0cm,3cm)  -- ++(-3cm,0cm) -- cycle}
  \draw (0,0) \rectanglepath;
  \draw (6.5,0) \rectanglepath;
  \draw [decorate,
    decoration = {calligraphic brace, mirror}] (0.1,0.1) --  (0.1,1.9);
 
 \draw [decorate,
    decoration = {calligraphic brace, mirror}] (1.1,1.6) --  (1.1,2.2);
    
     \draw [decorate,
    decoration = {calligraphic brace, mirror}] (2.1,1.8) --  (2.1,2.9);
 
\draw (1,0)--(1,3);
\draw (2,0)--(2,3);
\draw (0,2)--(1,2);
\draw (1, 1.5)--(2, 1.5);
\draw (1,2.3)--(2, 2.3);
\draw (2, 1.7)--(3,1.7);
\node (X) at (1.5,-0.5) {$\Delta^1$};
\node (Y) at (8,-0.5) {$\Delta^1$};
\node (Z) at (-0.5,1.5) {$M$};
\node (W) at (6,1.5) {$M$};
\node (R) at (0.58,1.) {\scalebox{0.5}{$\text{supp}(\mu_1)$}};
\node (R) at (1.58,1.9) {\scalebox{0.5}{$\text{supp}(\mu_2)$}};
\node (R) at (2.6,2.37) {\scalebox{0.5}{$\text{supp}(\mu_3)$}};
\draw [ fill=gray] (0,2) rectangle (1,3);
\draw [fill=gray] (1,0) rectangle (2,1.5);
\draw [ fill=gray] (1,2.3) rectangle (2,3);
\draw [ fill=gray] (2,0) rectangle (3,1.7);
\draw [ultra thick] (6.5,2) to [out=-20,in=150] (9.5,1.5);
\draw [ultra thick] (6.5,2.3) to [out=-30,in=150] (9.5,1.7);
\draw [ultra thick] (6.5,0) -- (6.5,3);
\draw [ultra thick] (9.5,0) -- (9.5,3);
\draw [->](4,1.5) -- (5.,1.5) node[midway,above] {$H_1$};
\end{tikzpicture}
\]
\caption{Fragmentation map for $N=3$ and $q=1$. The bold lines are the images of $M\times\{0\}, M\times \{1/3\}, M\times\{2/3\}$ and $M\times \{1\}$ under the map $H_1$.}\label{frag}
\end{figure}
As in Figure \ref{frag}, the map $H_1$ is defined so that the gray area  is mapped onto the union of  the bold lines in the target where the union of bold lines is a subcomplex of $M\times \Delta^q$ of dimension $n=\text{dim}(M)$. 

It is easy to check that $H_t$ is compatible with the face maps $d_i: \Delta^{q-1}\to \Delta^q$. Therefore, for any simplicial complex $K$, we still can define the homotopy $H_t: M\times K\to M\times K$.  Note that we can choose the integer $N$ as large as we want but as we shall see in the proof of \Cref{fragment} we want to homotope a map $g:K\to \text{\textnormal{Sect}}_c(\pi)$, and the choice of $N$ depends on the dimension of the parameter space $K$.

\begin{defn}To define the analogue of bold lines in Figure \ref{frag} for the simplicial complex $K$, let $V(\Delta^q)$ be the set ${\bf t}\in \Delta^q$ such that $N{\bf t}$ is a vector with integer coordinates. Let $V(K)$ be the union of $V(\Delta^q)$, where the union is taken over simplices of $K$. The analogue of bold lines is $L(K)=H_1(M\times V(K))\subset M\times K$. \end{defn}
Note that the topological dimension of the subcomplex $L$ is $n$. But if we choose any small open ball $B_{\epsilon}$ of radius say $2^{-q-1}\epsilon$ where $q=\text{dim}(K)$, then the homotopical dimension of $L_{\epsilon}(K):= H_1((M\backslash B_{\epsilon})\times V(K))$ is $n-1$ because the $n$-dimensional manifold $M\backslash B_{\epsilon}$ has homotopical dimension $n-1$ by which we mean it has the homotopy type of a CW complex of dimension $n-1$. 

The fragmentation map $H_1$ has the following useful property.
\begin{lem}\label{property}
Let $H_1:M\times \Delta^q\to M\times \Delta^q$ be the fragmentation map. For each ${\bf t}\in \Delta^q$, the space $(M\times {\bf t})\backslash H^{-1}_1(L_{\epsilon}(\Delta^q))$ can be covered by the support of at most $q$ functions among the partition of unity functions and the ball $B_{\epsilon}$. 
\end{lem}
\begin{proof}
This is straightforward from the definitions. As in Figure \ref{frag}, the complement of the gray area in each slice $M\times {\bf t}$ can be covered by the support of one function from
the chosen partition of unity. In general, the complement of  $H^{-1}_1(L(\Delta^q))$ in the slice $M\times {\bf t}$ can be covered by the support of at most $q$ functions (one for each coordinate of $\Delta^q$) among the partition of unity functions. Given that $H_1$ preserves the $M$ factor, to cover the complement of $H^{-1}_1(L_{\epsilon}(\Delta^q))$ in the slice $M\times {\bf t}$, we only need to add $B_{\epsilon}$.
\end{proof}
Now we want to use this lemma to prove \Cref{fragment}. To deform a family of sections of $\pi: E\to M$, parametrized by a map $g:K\to \text{\textnormal{Sect}}_c(\pi)$, we consider its adjoint as a map $G: M\times K\to E$. We also define the support of $g$ over $K$ with respect to the base section $s_0$ as follows.
\begin{defn}
Let $\text{supp}(g|_K)$  consist of the closure of those points $x\in M$ for which there exists at least one $t\in K$ such that $G(x,t)\neq s_0(x)$.
\end{defn}

 We shall need the following lemma that uses the fiber of the map $\pi:E\to M$ is $(n-1)$-connected to prove \Cref{fragment}.
\begin{lem}\label{lem1}
Given a family  $g:D^q\to \text{\textnormal{Sect}}_c(\pi)$, there exists a homotopy $g_s: D^q\to \text{\textnormal{Sect}}_c(\pi)$ so that for all $t\in D^q$ and $s\in [0,1]$, we have $\text{supp}(g_s(t))\subset \text{supp}(g(t))$ and at time $1$, the adjoint $G_1$ of $g_1$ satisfies $G_1(L_{\epsilon}(D^q))=s_0(M)$. 
\end{lem}
\begin{proof}
We think of the desired homotopy $G_t:M\times D^q\to E$ as a section of the pullback of $\pi: E\to M$ over $M\times D^q\times [0,1]$. The map $G_0$ is the adjoint of $g$. Let $Z\subset M\times D^q$ be the subcomplex consisting of points $(x,t)$ so that $G_0(x,t)=s_0(x)$. By the homotopy extension property, we will obtain the desired homotopy $G_t$, if we show that $G_0$ can be extended to a section $\tilde{G}$ over $$M\times D^q\times\{0\}\cup Z\times [0,1]\cup L_{\epsilon}(D^q)\times [0,1],$$ so that $\tilde{G}$ on $M\times D^q\times\{0\}$ is the same as $G_0$, on $Z\times [0,1]$ is given by $\tilde{G}(x,t,s)=s_0(x)$ and on $L_{\epsilon}(D^q)\times \{1\}$ is also given by $\tilde{G}(x,t,1)=s_0(x)$.  So far, we know how to define $\tilde{G}$ on $M\times D^q\times\{0\}\cup Z\times [0,1]$. We extend it over $L_{\epsilon}(D^q)\times [0,1]$ with a prescribed value on $ L_{\epsilon}(D^q)\times \{1\}$, by obstruction theory. Note that the homotopical dimension of $ L_{\epsilon}(D^q)\times [0,1]$ is $n$ and the fiber of the pullback of $\pi$ over $M\times D^q\times [0,1]$ is $(n-1)$-connected. Hence all obstruction classes that live in $H^*( L_{\epsilon}(D^q)\times [0,1]; \pi_{*-1}(\text{fiber}))$ vanish and we obtained the desired extension $\tilde{G}$. 
\end{proof}
\subsection{Proof of \Cref{fragment}} The idea is roughly as follows. To deform a family $g\colon D^q\to \text{Sect}_c(\pi)$ to a family of sections in $\text{Sect}_{\epsilon}(\pi)$, we  use \Cref{lem1} to assume that for the family $g$, we have $G(L_{\epsilon}(D^q))=s_0(M)$. We then use the fragmentation homotopy  to deform this family so that for each $s\in D^q$, the section $g(s)$ sends  ``most" part of $M$ to $G(L_{\epsilon}(D^q))$. For example in Figure \ref{frag}, for each $s\in D^1$, the support of the  section $g(s)$ lies inside the support of one function from
the partition of unity which can be chosen to be very small. 

 More precisely, we shall prove that homotopy groups of the pair $(\text{Sect}_c(\pi), \text{Sect}_{\epsilon}(\pi))$ are trivial. To do so, we show that for any commutative diagram
\begin{equation}\label{d1}
\begin{gathered}
\begin{tikzpicture}[node distance=2cm, auto]
  \node (A) {$S^{q-1}$};
  \node (B) [right of=A] {$\text{\textnormal{Sect}}_{\epsilon}(\pi)$};
  \node (C) [below of=A, node distance=1.2cm] {$ D^{q}$};  
  \node (D) [below of=B, node distance=1.2cm] {$ \text{\textnormal{Sect}}_c(\pi),$};
  \draw[->] (C) to node {$g$} (D);
  \draw [right hook->] (A) to node {}(C);
  \draw [->] (A) to node {$f$} (B);
  \draw [->] (B) to node {$$} (D);
\end{tikzpicture}
\end{gathered}
\end{equation}
there exists a homotopy of pairs $(g_t,f_t): (D^{q}, S^{q-1})\to (\text{Sect}_c(\pi), \text{Sect}_{\epsilon}(\pi))$ so that $f_0=f$, $g_0=g$ and $g_1:D^q\to \text{Sect}_c(\pi)$ factors through $\text{Sect}_{\epsilon}(\pi)$.  We first use the condition in \Cref{goodsection} to satisfy the following.
\begin{claim}
Note that for all $x\in S^{q-1}$, the support of $f(x)$ can be covered by at most $k$ balls of radius $2^{-k}\epsilon$ for some $k$. But we can also change $f$ up to homotopy to $f'$ such that for sufficiently fine triangulation of $S^{q-1}$, we can assume that for every simplex $\sigma\subset S^{q-1}$, we can cover $\text{\textnormal{supp}}(f'|_{\sigma})$  by  at most $k$ balls of radius $2^{-k}\epsilon$ for some $k$. 
\end{claim}
This is because there exists a fiberwise homotopy $h_t:E\to E$ that is the identity on $s_0(M)$ and whose time $1$ maps a neighborhood of $s_0(M)$ onto $s_0(M)$. So we can define a homotopy $F_t(x,s)=h_t(F(x,s)), G_t(x,s)=h_t(G(x,s))$ where $F$ and $G$ are adjoints of $f$ and $g$ respectively. These maps give a homotopy of the diagram \ref{d1} and it is easy to see for every $s\in S^{q-1}$, there exists a neighborhood $\sigma$ of $s$ so that $\text{supp}({F_1}|_{\sigma})\subset \text{supp}(f)(s)$. So from now on we assume that $f$ satisfies the claim. $\blacksquare$

To deform the family $g:D^q\to \text{Sect}_c(\pi)$ to a family in $\text{Sect}_{\epsilon}(\pi)$, we choose a partition of unity $\{\mu_i\}$ for a neighborhood  of $\text{supp}(g|_{D^q})$ so that  each $\text{supp}(\mu_i)$ can be covered by a ball of radius $2^{-q-1}\epsilon$. Let $H_t:M\times D^q\to M\times D^q$ be the fragmentation homotopy associated with this partition of unity. 

By \Cref{lem1}, there exists a homotopy $G': M\times D^q\times [0,1/2]\to E$ so that $G'_0$ is the adjoint of $g$, for all $s\in D^q$ and $t\in [0,1/2]$, we have $\text{supp}(G'_t(s))\subset \text{supp}(g(s)) $ and at time $1/2$, we have $G_{1/2}(L_{\epsilon}(D^q))=s_0(M)$. Note that if $s\in S^{q-1}$, then $G'_t(s)$ lies in $\text{Sect}_{\epsilon}(\pi)$. Therefore, $G'_t$ gives a homotopy of the pairs $(D^{q}, S^{q-1})\to (\text{Sect}_c(\pi), \text{Sect}_{\epsilon}(\pi))$.

Now we use the fragmentation homotopy to define $G_t: M\times D^q\to E$
\[
G_t:= \begin{cases}
G'_t & 0\leq t\leq 1/2,\\ G'_1\circ H_{2t-1} & 1/2 \leq t \leq 1. 
\end{cases}
\]

To show that $G_t$ is the desired homotopy, we first need to show that $G_t(-,S^{q-1})$ is also in $\text{Sect}_{\epsilon}(\pi)$ for $1/2 \leq t \leq 1$. Recall that by the claim, for every $x\in S^{q-1}$, there exists a simplex $\sigma$ containing $x$ so that $\text{supp}(f|_{\sigma})$ is contained in at most $k$ balls of radius $2^{-k}\epsilon$ for some $k$. We showed that  $\text{supp}(G'_1|_{\sigma})$ also has the same property.  Since the fragmentation homotopy preserves the $M$ factor, $\text{supp}(G'_1\circ H_{2t-1}|_{\sigma})$ also has the same property. Hence, $G_t(-,S^{q-1})$ lies in $\text{Sect}_{\epsilon}(\pi)$. So $G_t$ induces a homotopy of the pair of the map $(g,f)$. 

Now it is left to show that $G_1(-,s)$ lies in $\text{Sect}_{\epsilon}(\pi)$ for all $s\in D^q$. Note that the section $G_1(-,s)$ is the same as the base section on $H_1^{-1}(L_{\epsilon}(D^q))\cap M\times\{s\}$. Hence, by \Cref{property} the support of $G_1(-,s)$ can be covered by $q+1$ balls of radius $2^{-q-1}\epsilon$. Therefore, $G_1(-,s)$ is $\text{Sect}_{\epsilon}(\pi)$ for all $s\in D^q$.
\begin{rem} As we mentioned in the introduction, Morrison and Walker in their blob homology paper (\cite[Theorem 7.3.1]{MR2978449}) dropped the connectivity assumption but relaxed the notion of support to prove a key deformation lemma (\cite[Lemma B.0.4]{MR2978449}).  For a family $F:D^k\to  \text{\textnormal{Sect}}_c(\pi)$, they say $F$ is supported in $S\subset M$ if $F(p)(x)$ does not depend on $p$ for $x\notin S$. Our notion of support, however, requires $F(p)(x)$  to be equal to the value of the base section at $x$ for $x\notin S$. 

Note that when we drop the connectivity hypothesis, we no longer have \Cref{lem1}. However, for each ${\bf t} \in \Delta^q$, by \Cref{property}, we know that $(M\times {\bf t})\backslash H^{-1}_1(L_{\epsilon}(\Delta^q))$ is covered by at most $q$ open sets. Therefore, the same deformation $G_t$ as above, deforms a $\Delta^q$-family of sections to sections whose supports, in the sense of (\cite[Lemma B.0.4]{MR2978449}), can be covered by $q$ open balls.
\end{rem}
Note that $\text{\textnormal{Sect}}_{\epsilon}(\pi)$ which is a subspace of $ \text{\textnormal{Sect}}_c(\pi)$ has a natural filtration whose filtration quotients are similar to the filtration quotients induced by the non-abelian Poincar{\' e} duality (see \cite[Theorem 5.5.6.6]{lurie2016higher}). 

We in fact show in Appendix \ref{sec2} that this theorem implies the non-abelian Poincar{\' e} duality  for the space of sections of $\pi: E\to M$. To recall its statement,  let $\text{Disj}(M)$ be the  poset of the open subsets of $M$ that are homeomorphic to a disjoint union of finitely many open disks. For an open set $U\in \text{Disj}(M)$, let $\text{Sect}_c(U)$ denote the subspace of sections that are compactly supported and their supports are covered by $U$.  Although the non-abelian Poincar{\' e} duality holds for topological manifolds, to use the fragmentation idea, we assume that $M$ admits a metric for which there exists $\epsilon>0$ such that all balls of radius $\epsilon$ is geodesically convex. For example, this holds for all compact smooth manifolds.
\begin{cor}[Nonabelian Poincar{\' e} duality]\label{nonabelian}
If the fiber of the map $\pi$ is $(n-1)$-connected, the natural map
\[
\underset{U\in\text{Disj}(M)}{\textsf{hocolim }} \text{\textnormal{Sect}}_c(U)\to \text{\textnormal{Sect}}_c(\pi),
\]
is a weak homotopy equivalence. 
\end{cor}

\section{On h-principle theorems whose formal sections have highly connected fibers}\label{sec4}Let us recall the set-up from the introduction. Let $F: (\mathsf{Mfld}^{\partial}_n)^{op}\to \mathsf{S}$ be a topologically invariant sheaf in the sense of \cite[Section 2]{kupers2017three} from the category of smooth $n$-manifolds (possibly with nonempty boundary) with smooth embeddings as morphisms to a convenient category of spaces $\mathsf{S}$   (see \cite[Appendix A]{kupers2017three}). For our purpose, it is enough to consider the category of simplicial sets or compactly generated Hausdorff spaces. For brevity, when we refer to a simplicial set as a space, we mean the geometric realization of it.  Recall that we defined the space of formal solutions  $F^f(M)$ to be the space of sections of the bundle $\text{Fr}(M)\times _{\mathrm{GL}_n(\bR)} F^f(\bR^n)\to M$, where $\text{Fr}(M)$ is the frame bundle of $M$.  We say $F$ satisfies an h-principle if  the natural map from the functor to its homotopy sheafification (see \cite[Proposition 7.6]{MR3138384})
\[
j:F(M)\to F^f(M),
\]
 induces a weak equivalence and we say $F$ satisfies the c-principle if the above map induces a homology isomorphism. 
 
Often in proving h- and c-principles theorems, proving that the local statement $F(D^n)\xrightarrow{\simeq} F^f(D^n)$ which is a statement for $0$-handles is the easy step.   The hard step often is to inductively deduce the statement for higher handles relative to their attaching maps. Then one could prove the statement for compact manifolds using handle decompositions.  Thurston, however, proved a c-principle theorem in foliation theory (see \cite{mather2011homology} and \cite{sergeraert1979bgamma}) using his fragmentation idea without using the corresponding local statement. Proving the local statement in this c-principle theorem is surprisingly very subtle and it was later proved by Segal (\cite{segal1978classifying}) for smooth foliations and McDuff (\cite{mcduff1981groups}) for foliations with transverse volume form when the codimension is larger than $2$!
 
Let us first recall Thurston's theorem in this language. Let $F: (\mathsf{Mfld}^{\partial}_n)^{op}\to \text{sSet}$ be the functor from manifolds with a possibly non-empty boundary to simplicial sets so that the $q$-simplices  $F_{q}(M)$ is the set of codimension $n$ foliations on $M\times \Delta^q$ that are transverse to the fibers of  $M\times \Delta^q\to \Delta^q$. Let $F_c: \mathsf{Mfld}^{\partial}_n\to \text{sSet}$ be  {\it the compactly supported} version of $F$ meaning that we impose the condition that the foliations on $M\times \Delta^q$ are horizontal near the boundary $\partial M\times \Delta^q$. 
 
 Since in this case $F^f(M)$ is given by the section space of a bundle over $M$ whose fiber is $F^f(D^n)$, one could make sense of the compactly supported version by choosing a base section. In fact, there is a canonical choice of the base section so that we could define a map
 \[
j:F_c(M)\to F_c^f(M).
 \]
 Thurston uses his fragmentation technique on the closed disk $D^n$ to show directly (instead of induction on handles and inductively deloop) that $$|F_{c,\bullet}(\text{int}(D^n))|\to |F^f_{c,\bullet}(\text{int}(D^n))|$$ is a homology isomorphism. Recall that for the right-hand side, we have the weak homotopy equivalence $|F^f_{c,\bullet}(\text{int}(D^n))|\simeq \Omega^n|F^f_{\bullet}(D^n)|$.  
  
 Given the above delooping statement,  Thurston showed that this statement and the fragmentation on $M$ implies that $|F_{\bullet}(M)|\to |F^f_{\bullet}(M)|$ is a homology isomorphism for all compact manifolds $M$. If $M$ has a boundary, there is a version relative to the boundary. His fragmentation technique avoids the usual delooping steps in other approaches to go inductively from the statement for a handle of index $i$ to that of a handle of index $i+1$ and also avoids the step for $0$-handles. 
 
 
 To recall the main theorem, let $F$ be a topologically invariant sheaf enriched over $\mathsf{S}$  meaning that the sheaf is space valued and restriction maps are continuous. Suppose that there is a canonical base element in $F(N)$ for each manifold $N$ so that for a manifold with boundary $M$, we can define the relative version $F(M,\partial)$ to be the subspace of those elements in $F(M)$ that restrict to the base element in the germ of the boundary. We can also define {\it the compactly supported} version $F_c(M)$ to be the subspace of $F(M)$ consisting of those elements that restrict to the base element outside of a compact subset of $M$. Similarly, we can define the relative and compactly supported versions for $F^f$ so that we have a map $F_c(M)\to F_c^f(M)$. Similar to the previous section, we can define {\it $\epsilon$-supported versions} $F_{\epsilon}(M)$ and $F_{\epsilon}^f(M)$. We need to impose a homotopy theory condition on $F$ similar to the condition in \Cref{goodsection}. It is easy to see that this condition is satisfied for all geometric examples in the introduction and it will be necessary to find a simplicial resolution for $F$ in \Cref{resolution} and as we shall explain this is also a technical oversight in Mather's note. 
\begin{defn} \label{cond2} We say $F$ is well-pointed if for every manifold $M$ there exists a base point $s_0(M)\in F(M)$ and an open neighborhood $V_M$ of $s_0(M)$ such that 
\begin{itemize}
\item The open set $V_M$ deformation retracts to $s_0(M)$.
\item If $U\subset M$ is an open subset, the restriction $r\colon F(M)\to F(U)$ sends $s_0(M)$ to $s_0(U)$ and  $r({V_M})$ also is an open set that deformation retracts to $s_0(U)$.
\end{itemize}
\end{defn}
\begin{defn}\label{ha}
Given the base point $s_0(M)$ and the neighborhood $V_M$, we define the subspace $\widetilde{F_c(M,V_M)}\subset F(M)$ consisting of all elements $s$ such that there exists a compact set $K\subset M$ where the restriction of $s$ to $F(M\backslash K)$ lies in the restriction of $V_M$ to $M\backslash K$. These elements are said to be lax and compactly supported.
\end{defn}
Note that being compactly supported means that, for some compact set $K\subset M$, the restriction of $s$ to $F(M\backslash K)$ coincides with the base point. 
%
 \begin{defn}\label{condi}
We say that a well-pointed $F$ is {\it good}, if it satisfies 
\begin{enumerate}
\item The subspace of elements with empty support in $F(M)$ is contractible. 
\item There exists a neighborhood $V_M$ that deformation retracts to $s_0(M)$ such that the inclusion $F_c(M)\to \widetilde{F_c(M,V_M)}$ is a weak equivalence.
\item  Let $U$ be an open subset  of a manifold $M$ and let $r\colon F(M)\to F(U)$ be the restriction map. For all such $U$, the inclusion $\widetilde{F_c(U,r(V_M))}\to \widetilde{F_c(M,V_M)}$ be  an open embedding. 
\item For each finite family of open sets $U_0, U_1,\dots U_k$ such that $U_1,\dots, U_k$ are pairwise disjoint and contained in $U_0$, we have a permutation invariant map
\[
\mu^{U_0}_{U_1,\dots, U_k}\colon \prod_{i=1}^k F_c(U_i)\to F_c(U_0),
\]
where this map satisfies the obvious associativity conditions and for $U_0=\bigcup_{i=1}^{k} U_i$, the map $\mu^{U_0}_{U_1,\dots, U_k}$ is a weak equivalence.
\item Let $U$ and $V$ be open disks.  All embeddings $U\hookrightarrow V$ induces a homology isomorphism between $F_c(U)$ and  $F_c(V)$.
\item Let $\partial_1$ be the northern-hemisphere boundary of $D^n$.  Let $F(D^n, \partial_1)$ be the subspace of $F(D^n)$ that restricts to the base element in a germ of $\partial_1$ inside $D^n$. We assume $F(D^n, \partial_1)$ is contractible.  
\end{enumerate}
  \end{defn}
 \begin{thm}\label{h}
 Let $F$ be a good functor such that $F(D^n)\xrightarrow{\simeq} F^f(D^n)$. We assume that these spaces  are at least $(n-1)$-connected and $F$ has the fragmentation property meaning that 
 \[
 F_{\epsilon}(M)\to F_c(M),
 \]
 is a weak homotopy equivalence for every small enough $\epsilon>0$. Then for any compact manifold $M$, the map
 \[
 F_c(M)\to F_c^f(M),
 \]
 is a homology isomorphism. 
 \end{thm}
 \begin{exmp}
 Let $\Gamma_n^{\text{vol}}$ denote the topological Haefliger groupoid whose objects are $\bR^n$ with the usual topology and the space of morphisms are local volume-preserving diffeomorphisms of $\bR^n$ with respect to the standard volume form (see \cite{haefliger1971homotopy} for more details on how this groupoid is topologized). Let $\mathrm{B}\Gamma_n^{\text{vol}}$  denote its classifying space. There is a map 
\[
\theta: \mathrm{B}\Gamma_n^{\text{vol}}\to \mathrm{BSL}_n(\bR),
\]
which is induced by the functor $\Gamma_n^{\text{vol}}\to \mathrm{SL}_n(\bR)$ that sends a local diffeomorphism to its derivative at its source. We denote the homotopy fiber of $\theta$ by $\overline{\mathrm{B}\Gamma_n^{\text{vol}}}$. 
Let $M$ be an $n$-dimensional manifold with possibly non-empty boundary with a fixed volume form $\omega$.  Let $\tau^*(\theta)$ be the bundle over $M$ given by the pullback of $\theta$ via the map $\tau$

\begin{equation}\label{eq1}
\begin{gathered}
 \begin{tikzpicture}[node distance=2cm, auto]
  \node (C)  {$M$};  
  \node (D) [right of=C, node distance=2cm] {$\mathrm{BSL}_n(\bR).$};
    \node (B)  [above of=D, node distance=1.3 cm] {$\mathrm{B}\Gamma_{n}^{\text{vol}}$};

  \draw[->] (C) to node {$\tau$} (D);
  \draw [->] (B) to node {$\theta$} (D);
\end{tikzpicture}
\end{gathered}
\end{equation}
which is the classifying map for the tangent bundle. The space of sections of $\tau^*(\theta)$ has a natural base point $s_0$. Let $\text{Sect}(\tau^*(\theta),\partial)$ to be those sections that are equal to $s_0$ in the germ of the boundary (see \cite[Section 5.1]{nariman2014homologicalstability} for more details). It was proved by Haefliger that the fiber of $\tau^*(\theta)$ is $(n-1)$-connected. Note that $\text{Sect}(\tau^*(\theta),\partial)$ is not connected.

Let $\Diff_{\omega}(M,\partial)$ be the group of volume preserving diffeomorphisms with $C^{\infty}$-topology. And let $\dDiff_{\omega}(M,\partial)$ be the same group with the discrete topology. Now let $\overline{\BDiff_{\omega}(M,\partial)}$ denote the homotopy fiber of the natural map
\[
\BdDiff_{\omega}(M,\partial)\to \BDiff_{\omega}(M,\partial),
\]
induced by the identity homomorphism. This space can be thought of as the space of foliated $M$-bundles with a transverse volume form. It is easy to check the conditions in \Cref{condi} except the second condition which is proved by McDuff in \cite{MR699012}.  McDuff (\cite{mcduff1981groups}) showed that $\BdDiff_{\omega}(\bR^n)\to \mathrm{B}\Gamma_{n}^{\text{vol}}$ is a homology isomorphism for $n>2$ and it still not known for $n=2$. So the local statement in this case is known for $n>2$. She used this fact to show that when $\text{dim}(M)>2$
\[
\overline{\BDiff_{\omega}(M,\partial)}\to \text{Sect}(\tau^*(\theta),\partial),
\]
 induces a homology isomorphism into the connected component that it hits. She also found a different proof for $\text{dim}(M)=2$ in \cite{MR707329}. But by using \Cref{h}, one could give a uniform proof for the compactly supported version without using her local statement in dimension $3$ and higher. 
 \end{exmp}
   
\begin{exmp}\label{conf} Let $M$ be a manifold of positive dimension and let $F(M)$  be the labeled configuration space (\cite{MR922926}, \cite{segal1973configuration}) for which proving the fragmentation property is easy. To recall the definition from \cite{MR922926}, let $X$ be a fixed connected CW complex with a base point $x_0$. Let $C(M;X)$ be the configuration space of a finite number of distinct points in $M$ with labels in $X$ and the topology is such that points can vanish if their label is $x_0$ (For a precise definition of the topology see \cite{MR922926}, \cite{segal1973configuration}). We shall write a point $\xi\in C(M;X)$ as a formal sum $\sum x_im_i$ where $m_i\in M$ are distinct points and $x_i\in X$ satisfying the relation $\sum x_im_i\sim \sum x_im_i+x_0m$. For a subspace $N\subset M$, we let $C(M,N;X)$ be the quotient of $C(M;X)$ by the relation $\sum x_im_i\sim \sum x_im_i+xn$ where $n\in N$. We define the support of $\sum x_im_i$ to be the set of the points $m_i$ whose label $x_i$ is not the base point $x_0$. Note that, similar to section spaces, we can define the subspace $C_{\epsilon}(M;X)$ to be that labeled configuration of points whose support can be covered by $k$ balls of radius $2^{-k}\epsilon$ for some $k$. But obviously we have $C_{\epsilon}(M;X)=C(M;X)$.

 It is easy to show that $C(D^n,\partial D^n; X)$ is homotopy equivalent to the reduced suspension $\Sigma^nX$ which is at least $n$-connected.  The fragmentation method implies that the natural scanning map (see \cite[Definition 6.3.5]{knudsen2018configuration})
 \[
 C(D^n;X)\to \Omega^nC(D^n,\partial D^n; X),
 \] 
 is a homology isomorphism (it is in fact a weak homotopy equivalence by \cite{segal1973configuration}). Using fragmentation again for $C(M;X)$, we could obtain the homological version of McDuff's theorem (\cite{MR0358766}) that for any closed manifold $M$, the natural map $$C(M;X)\to\text{Sect}_c(\text{Fr}(M)\times _{\mathrm{GL}_n(\bR)} \Sigma^nX\to M),$$ induces a homology isomorphism. 
\end{exmp}

\subsection{$n$-fold delooping via fragmentation}\label{disk} The key step in proving \Cref{h} is to show that if $F$ has a fragmentation property then the map
\[
F(D^n,\partial)\to F^f(D^n,\partial)\simeq \Omega^nF(D^n),
\]
is a homology isomorphism. To do so, we filter $F(D^n)$ and $F^f(D^n)$. Since $D^n$ is compact, the fragmentation property for $F$ and $F^f$ implies that 
\[
F_{\epsilon}(D^n)\xrightarrow{\simeq} F(D^n),
\]
\[
F^f_{\epsilon}(D^n)\xrightarrow{\simeq} F^f(D^n).
\]
The spaces $F_{\epsilon}(D^n)$ and $F^f_{\epsilon}(D^n)$ are naturally filtered by the number of balls that cover the supports. We shall denote these filtrations and the corresponding maps between them by
\begin{equation}\label{R}
  \begin{gathered}
\begin{tikzcd}
F_1(D^n) \arrow{d}{j_1}\arrow[r,hook] &F_2(D^n) \arrow{d}{j_2}\arrow[r,hook] &\cdots \arrow[r,hook] & F_{\epsilon}(D^n) \arrow{d}{j}\arrow{r}{\simeq}&F(D^n)\arrow{d}{\simeq}\\
F^f_1(D^n)\arrow[r,hook] &F^f_2(D^n) \arrow[r,hook] &\cdots \arrow[r,hook] & F^f_{\epsilon}(D^n)\arrow{r}{\simeq}&F^f(D^n).
\end{tikzcd}
\end{gathered}
\end{equation}
Note that the last vertical map is a weak equivalence because $F^f(D^n)$ is a section space of a bundle over contractible space $D^n$ with the fiber $F(D^n)$.  Therefore, the map $j$ in the diagram \ref{R} also is a weak homotopy equivalence. 
\begin{rem}
We dropped $\epsilon$ from our notations for filtrations $F_k(-)$ and $F^f_k(-)$ but if we want to emphasize our choice of $\epsilon$, we shall instead use $F_k(-,\epsilon)$ and $F^f_k(-,\epsilon)$.
\end{rem}

\begin{prop}\label{j_1}
Let $F$ be a good functor satisfying the hypothesis of \Cref{h}. Now if $j_1$ in the diagram \ref{R} induces a homology isomorphism, so does the map 
\[
F(D^n,\partial)\to F^f(D^n,\partial)\simeq \Omega^nF(D^n).
\]
\end{prop}
We first explain the strategy to prove that $j_1$ is a homology isomorphism before we embark on proving \Cref{j_1}. We have the following general lemma about filtered spaces (\cite[Lemma 2, Section 27]{mather2011homology}):
\begin{lem}\label{blueprint}
Consider the commutative diagram of spaces
\[
\begin{tikzcd}
X_1 \arrow{d}{f_1}\arrow[r,hook] &X_2 \arrow{d}{f_2}\arrow[r,hook] &\cdots \arrow[r,hook] & X_{\infty} \arrow{d}{f_{\infty}}\arrow{r}{\iota}&X\arrow{d}{f}\\
Y_1\arrow[r,hook] &Y_2 \arrow[r,hook] &\cdots \arrow[r,hook] & Y_{\infty}\arrow{r}{\iota'}&Y.
\end{tikzcd}
\]
Suppose:
\begin{itemize}
\item $X_{\infty}$ and $Y_{\infty}$ are the union of $X_i$'s and $Y_i$'s respectively and for each $i$, the pairs $(X_i, X_{i-1})$ and $(Y_i, Y_{i-1})$ are good pairs.\footnote{The pair $(A,B)$ of topological spaces where $B\subset A$ is a good pair if there exists an open neighborhood of $B$ in $A$ such that it deformation retracts to $B$.}
\item $f$, $\iota$ and $\iota'$ are weak homotopy equivalences.
\item  The filtration is so that if $f_1$ is $k$-acyclic \footnote{we say $f:A\to B$ is $k$-acyclic if it induces a homology isomorphism  for homological degrees less than $k$ and surjection on degree $k$} for some $k$ then the induced map 
\[
\overline{f_N}:X_N/X_{N-1}\to Y_N/Y_{N-1},
\]
is $(2N+k-2)$-acyclic for every integer $N>1$.
\end{itemize}
Then $f_1$ induces a homology isomorphism.
\end{lem}
\begin{proof}
We can assume that the maps $f_i$ are inclusions by replacing them with the mapping cylinder of $f_i$.  Therefore, the filtration $(Y_p,X_p)$ of $(Y_{\infty},X_{\infty})$ gives rise to a spectral sequence whose first page is $$E^1_{p,q}=H_{p+q}(Y_p, Y_{p-1}\cup X_p).$$ It converges to the homology of the pair $(Y_{\infty},X_{\infty})$ but this pair is weakly homotopy equivalent to the pair $(Y,X)$. Since the first condition $f$ is a weak homotopy equivalence, the spectral sequence converges to zero.  \begin{figure}[h]
\[
\begin{tikzpicture}
\draw [<-] (-0.6,4.7) -- (-0.6,0.2);
\draw [->] (-1,0.6)-- (4.6,0.6);
\node  at (0.2,0.9) {\small$0$};
\node  at (0.2,0.3) {$1$};
\node at (1.2,0.9) {\small$0$};
\node at (1.2,0.3) {\small$2$};

\node at (2.5,0.9) {\small$0$};
\node at (2.5,0.3) {\small$3$};

\node at (3.7,0.9) {\small$0$};
\node at (3.7,0.3) {\small$4$};



\node at (4.8,0.3) {\small$p$};

\node at (0.2,1.5) {\small$\vdots$};
\node at (-1,1.5) {\small$\vdots$};
\node at (-1,0.9) {\small$0$};

\node at (0.2,2.2) {\small$0$};
\node at (-1,2.2) {\small$k-2$};

\node at (0.2,3.2) {\small$0$};
\node at (-1,3.2) {\small$k-1$};

\node at (0.2,4.2) {\small$*$};
\node at (-1,4.2) {\small$k$};

\node at (0.2,5.2) {\small$$};
\node at (-1,4.9) {\small$q$};

\node at (1.2,2.2) {\small$0$};
\node at (2.5,2.2) {\small$0$};
\node at (3.7,2.2) {\small$0$};
\node at (1.2,3.2) {\small$0$};
\node at (2.5,3.2) {\small$0$};
\node at (1.2,4.2) {\small$0$};

\node at (1.2,1.5) {\small$\vdots$};
\node at (2.5,1.5) {\small$\vdots$};
\node at (3.7,1.5) {\small$\vdots$};
\draw [-,red] (3.6,2.3)--(0.3,4.2);
\draw [->,red] (2.3,3.2)--(0.3,4.2);
\draw [-,red] (1.1,4.2)--(0.3,4.2);

\end{tikzpicture}\]
\caption{All the differentials that map to $E^1_{1,k}$ have trivial domains. We drew differentials on the first, second, and third pages.}\label{s}
\end{figure}
Now we suppose the contrary that $f_1$ is not a homology isomorphism and we choose the smallest $k$ so that $E^1_{1,k}=H_{k+1}(Y_1,X_1)\neq 0$. Therefore, $f_1$ is $k$-acyclic and by the third condition $\overline{f_p}$ is $(2p+k-2)$-acyclic which implies that $E^1_{p,q}=H_{p+q}(Y_p, Y_{p-1}\cup X_p)=0$ for $q\leq p+k-2$. 

Hence, as is indicated in Figure \ref{s}, no nontrivial differentials can possibly hit $E^1_{1,k}$ which contradicts the fact that the spectral sequence converges to zero in all degrees.
\end{proof}
In order to apply \Cref{blueprint} to the diagram \ref{R}, we need to establish the second condition of \Cref{blueprint} for the diagram \ref{R}. The subtlety here is in the filtrations $F_k(-)$ and $F_k^f(-)$ where we know that the support is covered by $k$ small balls but the data of these balls are not given.  We shall define certain auxiliary spaces by adding the data of covering balls. 
\subsubsection{Semisimplicial resolutions}\label{ssSpace} To study the  filtration quotients in the diagram \ref{R}, we shall define auxiliary semisimplicial spaces. 

For the definition of semisimplicial spaces and the relevant techniques, we follow \cite{ebert2017semi}. Briefly, what we need about semisimplicial spaces and their (fat) realizations are as follows:  first, a semisimplicial map that is a weak homotopy equivalence in each degree induces a weak homotopy equivalence between fat realizations (\cite[Theorem 2.2]{ebert2017semi}), second, there is a skeletal filtration on the fat realization that gives rise to a spectral sequence calculating the homology of the fat realization (\cite[Section 1.4]{ebert2017semi}), and the last is the technical lemma in \cite[Proposition 2.8]{galatius2018homological} that gives a useful criterion to prove that the augmentation map for an augmented semisimplicial space induces a weak homotopy equivalence after taking realizations.

 Recall that  we assumed that $M$ is a geodesic space and a small positive $\epsilon$ exists so that all balls of radius $\epsilon$ are geodesically convex. We say that a subset $U$ of $M$ is {\it $\epsilon$-admissible} if it is open, geodesically convex and it can be covered by an open ball of radius $\epsilon$.
\begin{defn}\label{O}
Let $\mathcal{O}_{\epsilon}(M)$ be the discrete poset of open subsets of $M$ that can be covered by a union of $k$ geodesically convex balls of radius at most $2^{-k}\epsilon$ for some positive integer $k$.
\end{defn}
\begin{defn}
Let $CF_k(M)$ be the subspace of $F(M)^k$ consisting of $k$-tuples so that each one has support contained in one ball of radius $2^{-k}\epsilon$. We define the subspace $DF_k(M)$ of $CF_k(M)$ to be {\it degenerate} $k$-tuples  that is the union of their supports can be covered by $k_0$ balls of radius $2^{-k_0}\epsilon$ for some $k_0<k$. We denote the quotient space $CF_k(M)/DF_k(M)$ by $NF_k(M)$. Similarly, we can define $CF^f_k(M)$, $DF^f_k(M)$ and $NF^f_k(M)$. 
\end{defn}

The natural maps $NF_k(M)\to F_k(M)/F_{k-1}(M)$ and $NF^f_k(M)\to F^f_k(M)/F^f_{k-1}(M)$ are $(k!)$-sheeted covers away from the base points. So if $\Sigma_k$ denotes the permutation group on $k$ letters, we have the spectral sequence of the action whose $E^2$-page is $H_p(\Sigma_k;H_q( NF_k(M)))$ converging to $H_{p+q}(F_k(M)/F_{k-1}(M))$. Similarly, we have the same spectral sequence for $NF^f_k(M)$ and the comparison of the spectral sequences implies the following.
\begin{lem}\label{filtrationquotients}
If the induced map $NF_k(M)\to NF^f_k(M)$ is $j$-acyclic so is the map between the filtration quotients
\[
F_k(M)/F_{k-1}(M)\to F^f_k(M)/F^f_{k-1}(M).
\]
\end{lem}
Hence, to establish the third condition of \Cref{blueprint} for the diagram \ref{R}, it is enough to study the acyclicity of the map $NF_k(M)\to NF^f_k(M)$. To do so, we shall use the following semisimplicial spaces.
\begin{defn}
 Let ${CF_k(M)}_{\bullet}$ be a semisimplicial space whose space of $q$-simplices  is given by the tuples $(\sigma, (B_{ij}))$ where $\sigma=(\sigma_1,\dots,\sigma_k)\in CF_k(M)$ and $(B_{ij})$ is a $k\times (q+1)$  matrix of $(2^{-k}\epsilon)$-admissible sets such that $B_{ij}$ contains the support of $\sigma_i$ for all $j$ (if the support of $\sigma_i$ is empty $B_{ij}$'s are just $(2^{-k}\epsilon)$-admissible sets). We topologize the $q$-simplices as a subspace of $F(M)^k\times \mathcal{O}_{\epsilon}(M)^{kq+k}$.
\end{defn}
\begin{defn}\label{Delta}
 We define subsemisimplicial space $D{F_k(M)}_{\bullet}$ so that its $q$-simplices are given by pairs $(\sigma, (B_{ij}))$ so that for each $0\leq j\leq q$, the closure of $\cup_i B_{ij}$ is covered by $k_0$ balls of radius $2^{-k_0}\epsilon$ for some $k_0<k$.
\end{defn}
We define similarly ${CF^f_k(M)}_{\bullet}$ and $D{F^f_k(M)}_{\bullet}$.
\begin{rem}\label{epsilon}
If we keep track of  the choice of $\epsilon$ in our notations we have the  useful identifications ${CF_k(M,\epsilon)}_{\bullet}={CF_1(M,2^{-k}\epsilon)}_{\bullet}^k$ and the same for $F^f$. 
\end{rem}
\begin{lem}\label{resolution}
The natural maps
\[
||{CF_k(M)}_{\bullet}||\to CF_k(M), \, ||D{F_k(M)}_{\bullet}||\to DF_k(M),
\]
are all weak homotopy equivalencies where $||-||$ means the fat realization of a semi-simplicial space (see \cite{ebert2017semi}). Similarly, the corresponding statement holds for $F^f$.
\end{lem}
\begin{proof}
This is  \cite[Lemma in section 20]{mather2011homology} for the functor defined by Thurston. But there is an oversight in that proof that Mather assumes that the augmentation map from the realization of semisimplicial sets to $F_k(M)$ is a fibration and says that it is enough to show that their fibers are contractible. To fix this oversight, we need \Cref{condi}. The idea is to show that the augmentation maps are microfibrations with contractible fibers. 

Let $\widetilde{CF_k(M, V_M)}$ be the subspace of $F(M)^k$ consisting of $k$-tuples such that each one has a lax support (see \Cref{ha})  in ball of radius $2^{-k}\epsilon$. Similarly, we define $\widetilde{CF_k(M,V_M)}_{\bullet}$.  Hence, it is enough to show that 
\[
\alpha\colon||\widetilde{CF_k(M,V_M)}_{\bullet}||\to \widetilde{CF_k(M, V_M)},
\]
is a weak homotopy equivalence. 

Let $S_{\bullet}$ be the simplicial set whose $q$-simplices are given by $q+1$ ordered  $(2^{-k}\epsilon)$-admissible sets. By the third condition of the goodness of the functor (see \Cref{condi}), it is clear that $\widetilde{CF_k(M,V_M)}_{\bullet}\subset\widetilde{CF_k(M,V_M)}\times (S_{\bullet})^k$ is open. Similar to the proof of \Cref{micro},  this inclusion satisfies the conditions of \cite[Proposition 2.8]{galatius2018homological}. Therefore, the map $\alpha$
induced by the projection to the first factor is microfibration. To identify the fiber over $\sigma=(\sigma_1,\dots,\sigma_k)$, let $S_i$ be the set of $(2^{-k}\epsilon)$-admissible sets containing the support of $\sigma_i$.  Let $S_{i\bullet}$ be the simplicial set whose $q$-simplices are given by mappings $[q]=\{0,1,\dots,q\}$ to $S_i$. Therefore, the realization of this simplicial set is contractible. The fiber over $\sigma$ can be identified with the fat realization of $S_{1\bullet}\times\cdots S_{k\bullet}$. Since the fat realization and the realization for the simplicial sets are weakly equivalent and the realization commutes with products (\cite{MR0084138}), we deduce that the fiber over $\sigma$ is contractible. The proof for the other augmentation map is similar. 
\end{proof}
Now the strategy to check the third condition of \Cref{blueprint} for the diagram \ref{R} is as follows. We define a functor $\nu_N$ on spaces so that when we apply it to a $k$-acyclic map $f:X\to Y$, we obtain a $(2N+k-2)$-acyclic map $\nu_N(f):\nu_N(X)\to \nu_N(Y)$. And then we construct a homotopy commutative diagram
\[ 
\begin{tikzpicture}[node distance=4cm, auto]
  \node (A) {$||{CF_k(D^n)}_{\bullet}||/||D{F_k(D^n)}_{\bullet}||$};
  \node (B) [right of=A] {$\nu_k(F(D^n,\partial))$};
  \node (C) [below of=A, node distance=1.7cm] {$ ||{CF^f_k(D^n)}_{\bullet}||/||D{F^f_k(D^n)}_{\bullet}||$};  
  \node (D) [below of=B, node distance=1.7cm] {$ \nu_k(F^f(D^n,\partial)),$};
  \draw[->] (C) to node {$$} (D);
  \draw [->] (A) to node {}(C);
  \draw [->] (A) to node {$$} (B);
  \draw [->] (B) to node {$$} (D);
\end{tikzpicture}
\]
where the horizontal maps induce  homology isomorphisms.  In the next section, we shall define a suitable functor $\nu_k$ satisfying the desired properties.
\subsubsection{A thick model of the suspension of a based space} To define the functor $\nu_k$ that receives a map from the above semisimplicial resolutions, we need to modify the definition of the suspension of a space. First, we define auxiliary simplicial sets associated with the manifold $M$ with the fixed choice of $\epsilon$.
\begin{defn}
Let $S(r)$ be the set of $(2^{-r}\epsilon)$-admissible sets in $M$. Let $\Delta_{\bullet}(M,r)$ denote the simplicial set whose $q$-simplices are given by mappings $[q]$ into $S(r)$ i.e. $(q+1)$-tuple of elements in $S(r)$. Let $M_{\bullet}(r)$ be the subsimplicial set of $\Delta_{\bullet}(M,r)$ whose $q$-simplices consist of those admissible sets that the intersection of the entries of the tuple is nontrivial. And let $\partial M_{\bullet}(r)$ be the subsimplicial set of $\Delta_{\bullet}(M,r)$ whose $q$-simplices consist of those admissible sets that the intersection of the entries of the tuple and $\partial M$ is nontrivial. 
\end{defn}
\begin{rem}
For each $k$, the geometric realizations of $\Delta_{\bullet}(M,r)$ are contractible because it is a full simplex, and the geometric realizations of $M_{\bullet}(r)$ and $\partial M_{\bullet}(r)$ by the nerve theorem have the homotopy type of $M$ and $\partial M$ respectively. 
\end{rem}
Our modification of the suspension of a space $X$ is
\begin{defn}
Let $\widetilde{\Sigma}^nX$ be the realization of the following semisimplicial space 
\[
\widetilde{\Sigma}_{\bullet}^nX(r):=\frac{(\Delta_{\bullet}(D^n,r)\times \{*\})\cup (D_{\bullet}^n(r)\times X)}{(t,x)\sim (t,x')\text{ if } t\in \partial D_{\bullet}^n(r)}.
\]
\end{defn}
Note that for each $r$, the space $\widetilde{\Sigma}^nX$ has the same homotopy type of the suspension $\Sigma^nX$ so we do not write the dependence on $r$. Because $D_{\bullet}^n(r)$ and $\partial D_{\bullet}^n(r)$ are semisimplicial sets that realize to the disk $D^n$ and the sphere $S^{n-1}$ respectively. And $\Delta_{\bullet}(D^n,r)\times \{*\}$ is a contractible semi-simplicial set that is glued to the base point. Note that we also have a natural projection $\pi:\widetilde{\Sigma}^nX(r)\to ||\Delta_{\bullet}(D^n,r)||$.
\begin{defn}\label{T}
Let $T_{k,\bullet}(M)$ be the subsimplicial set of $(\Delta_{\bullet}(M,k))^k$ whose $q$-simplices are given by matrices $(B_{ij})$, $i=0,1,\dots,q$, $j=1,\dots,k$ of admissible sets so that for each $i$, the union $\cup_j B_{ij}$ can be covered by $k_0$ open balls of radius $2^{-k_0}\epsilon$ for some $k_0<k$. For $k=1$, we define $T_{1,\bullet}=*$.
\end{defn}
\begin{defn}
We define $\theta_k(X)$ to be the pair $$((\widetilde{\Sigma}^nX)^k, (\pi^k)^{-1}(|T_{k,\bullet}(D^n)|)),$$ where $\pi^k:  (\widetilde{\Sigma}^nX)^k\to ||\Delta_{\bullet}(D^n, k)||^k$ is the natural projection. Let $\nu_k(X)$ denote the quotient 
\[
(\widetilde{\Sigma}^nX)^k/(\pi^k)^{-1}(|T_{k,\bullet}(D^n)|).
\]
\end{defn}
\begin{rem}\label{k=1}
Note that for $k=1$, the space $\nu_1(X)$ has the homotopy type of $\Sigma^nX$.
\end{rem}
We suppress $n$, the dimension from the notations $\theta_k(X)$  and $\nu_k(X)$ as  it is fixed throughout.  
The following technical lemma is the main property of the functor $\nu_k$.
\begin{lem}\label{acyc}
If $f:X\to Y$ is $j$-acyclic, the induced map of pairs $\nu_k(f): \nu_k(X)\to \nu_k(Y)$ is $(j+n+2k-2)$-acyclic. 
\end{lem} 
\begin{proof} Recall that the reduced suspension of $X$ for a based space $(X,*)$, is the smash product $S^n\wedge X$ and we represent points in this smash product by a pair $(s,x)$ where $s\in S^n $ and $x\in X$. First it is not hard to see (\cite[Section 24]{mather2011homology}) that  the space $\nu_k(X)$ is homotopy equivalent to 
\[
(S^n\wedge X)^k/\Delta_{\text{\textnormal{fat}},k}(S^n,X),
\]
where $\Delta_{\text{\textnormal{fat}},k}(S^n,X)$ consists of tuples $\big( (s_1,x_1), (s_2, x_2),\dots, (s_k, x_k)\big )$ such that $s_i=s_j$ for some $i\neq j$. We can further simplify the homotopy type of $\nu_k(X)$ by separating $S^n$ and $X$ in the above quotient to obtain
\[
\nu_k(X)\simeq \big (S^{nk}/\Delta_{\text{\textnormal{fat}},k}(S^n)\big )\wedge X^{\wedge k}.
\]
Note that if $f:X\to Y$ is $j$-acyclic, the long exact sequence for the homology of a pair implies that the induced map $f^{\wedge 2}: X\wedge X\to Y\wedge Y$ is $(j+1)$-acyclic. Hence, one can inductively show that the induced map $f^{\wedge k}:X^{\wedge k}\to Y^{\wedge k}$ is $(j+k-1)$-acyclic. Thus, it is enough to show that $\big (S^{nk}/\Delta_{\text{\textnormal{fat}},k}(S^n)\big )$ is $(n+k-2)$-acyclic. Using again the long exact sequence for the homology of a pair, we need to show that $\Delta_{\text{\textnormal{fat}},k}(S^n)$ is $(n+k-3)$-acyclic.

For $i\neq j$, let $\Delta_{(i,j)}(S^n,k)\subset (S^n)^{\wedge k}$ be the subspace given by tuples $(s_1,s_2,\dots,s_k)$ where $s_i=s_j$.  Note that $\Delta_{(i,j)}(S^n,k)\simeq (S^n)^{\wedge (k-1)}$. The fat diagonal $\Delta_{\text{\textnormal{fat}},k}(S^n)$ is the union of $\Delta_{(i,j)}(S^n,k)\subset (S^n)^{\wedge k}$ for all pairs $(i,j)$ where $i\neq j$. These are not open subsets but they are sub-CW complexes, so we still can apply Mayer-Vietoris spectral sequence for this cover to compute the homology of $\Delta_{\text{\textnormal{fat}},k}(S^n)$. Let $\Delta_{(i_1,j_1),\dots, (i_r,j_r)}(S^n,k)$ denote the intersection $\Delta_{(i_1,j_1)}(S^n,k)\cap\dots\cap \Delta_{(i_r,j_r)}(S^n,k)$. Hence, we have 
\[
E_{p,q}^1=\bigoplus_{(i_m,j_m)}H_q(\Delta_{(i_0,j_0),\dots, (i_p,j_p)}(S^n,k))\Longrightarrow H_{p+q}(\Delta_{\text{\textnormal{fat}},k}(S^n)),
\]
where the sum is over different tuples of pairs $(i_m,j_m)$. Since the intersection  $\Delta_{(i_0,j_0),\dots, (i_p,j_p)}(S^n,k)$ is a $n(k-p-1)$-connective space, $E^1_{p,q}=0$ for $q<n(k-p-1)$. Note that $p$ is at most $k-2$ so we have $n+k-3< n(k-p-1)+p$. On the other hand, if $p+q<n(k-1)-pn+p$, we have $E^1_{p,q}=0$ which implies that $\Delta_{\text{\textnormal{fat}},k}(S^n)$ is $(n+k-3)$-acyclic. 
\end{proof}
Now we are ready to prove the third condition of \Cref{blueprint} for the diagram \ref{R}.
\subsubsection{Proof of \Cref{h} for $M=D^n$}\label{diskcase} So we want to prove that the natural map 
\[
F(D^n,\partial)\to F^f(D^n,\partial)\simeq \Omega^nF(D^n),
\] 
induces a homology isomorphism. To do this, we show that 
\begin{lem}\label{sss} There exists a commutative diagram of pairs
\begin{equation}\label{RR}
  \begin{gathered}
\begin{tikzpicture}[node distance=4.3cm, auto]
  \node (A) {$(||{CF_k(D^n)}_{\bullet}||,||D{F_k(D^n)}_{\bullet}||)$};
  \node (B) [right of=A] {$\theta_k(F(D^n,\partial))$};
  \node (C) [below of=A, node distance=1.7cm] {$ (||{CF^f_k(D^n)}_{\bullet}||,||D{F^f_k(D^n)}_{\bullet}||)$};  
  \node (D) [below of=B, node distance=1.7cm] {$ \theta_k(F^f(D^n,\partial)),$};
  \draw[->] (C) to node {$$} (D);
  \draw [->] (A) to node {}(C);
  \draw [->] (A) to node {$$} (B);
  \draw [->] (B) to node {$$} (D);
\end{tikzpicture}
\end{gathered}
\end{equation}
so that the horizontal maps are homology isomorphisms (by which we mean homology isomorphism on each member of the pair). 
\end{lem}
Before we prove this lemma, let us explain how the above lemma finishes the proof of  \Cref{h} for $M=D^n$. By \Cref{resolution}, \Cref{acyc} and \Cref{sss}, the third condition of \Cref{blueprint} for the diagram \ref{R} holds. Hence, $j_1$ in the diagram \ref{R} is a homology isomorphism. Recall that for $k=1$, the pair $\theta_1(X)$ has the homotopy type of $(\Sigma^nX,*)$. Therefore, \Cref{j_1} follows from \Cref{sss} for $k=1$. 

\begin{con}To define the horizontal map in the diagram \ref{RR}, we first define a semi-simplicial map
\[
f_{\bullet}:{CF_1(D^n, 2^{-k}\epsilon)}_{\bullet}\to \widetilde{\Sigma}_{\bullet}^nF(D^n,\partial)(k)= \frac{(\Delta_{\bullet}(D^n, k)\times\{*\})\cup(D_{\bullet}^n(k)\times F(D^n,\partial))}{(t,x)\sim (t,x')\text{ if } t\in \partial D_{\bullet}^n(k)}.
\]
For a $q$-simplex $(\sigma, B_0,\dots,B_q)$ on the left hand side, we know that $\cap_i B_i$ contains $\text{supp}(\sigma)$. 

\begin{itemize}
\item If $(B_0,\dots, B_q)$ is a $q$-simplex in $\partial D^n_{\bullet}(k)$, then we send $(\sigma, B_0,\dots,B_q)$ to the base point on the right hand side. 
\item If $(B_0,\dots, B_q)$ is a $q$-simplex in $D^n_{\bullet}(k)$ but not in $\partial D^n_{\bullet}(k)$, then the support of $\sigma$ lies inside $D^n$. Therefore $\sigma\in F(D^n,\partial)$, so we send $(\sigma, B_0,\dots,B_q)$ to the corresponding element in $ F(D^n,\partial)\times D_{\bullet}^n(k)$. 
\item And if $(B_0,\dots,B_q)$ is in $\Delta_{\bullet}(D^n,k)$ but not in $D_{\bullet}^n(k)$, we send $(\sigma, B_0,\dots,B_q)$, to $(B_0,\dots,B_q)$ in $\Delta_q(D^n,k)\times \{*\}$. 
\end{itemize} 
Since the above map is a semisimplicial map, we could  take the realization to obtain 
\begin{equation}\label{*}
f:||{CF_1(D^n, 2^{-k}\epsilon)}_{\bullet}||\to \widetilde{\Sigma}^nF(D^n,\partial).
\end{equation}

Recall from \Cref{epsilon}, that ${CF_k(M,\epsilon)}_{\bullet}={CF_1(M,2^{-k}\epsilon)}^k$. Therefore, the above construction gives rise to maps
\[
||{CF_k(D^n)}_{\bullet}||\to (\widetilde{\Sigma}^nF(D^n,\partial))^k.
\]
\Cref{T} and \Cref{Delta} are so that the above map induces a map
\[
||D{F^f_k(D^n)}_{\bullet}||\to (\pi^k)^{-1}(||T_{k,\bullet}(D^n)||).
\]
The lower horizontal map in diagram \ref{RR} is similarly defined. 
\end{con}
\begin{proof}[Proof of \Cref{sss}] From the naturality of the construction, we obtain the commutative diagram \ref{RR}. We show that the top horizontal map is a homology isomorphism. The proof for the bottom horizontal map is similar. We first show that the map $f$ in \ref{*} induces a homology isomorphism. Recall that for all semisimplicial spaces $X_{\bullet}$ , there is a spectral sequence $E^1_{p,q}(X_{\bullet})=H_q(X_p)$ that converges to $H_{p+q}(||X_{\bullet}||)$. The map $f$ induces a comparison map between spectral sequence
\begin{equation}\label{eq:5}
\begin{gathered}
\begin{tikzcd}
H_q({CF_1(D^n, 2^{-k}\epsilon)}_{p})\arrow["(f_p)_*"]{r}\arrow[Rightarrow]{d}&H_q(\widetilde{\Sigma}_p^nF(D^n,\partial))\arrow[Rightarrow]{d}\\H_{p+q}(||{CF_1(D^n, 2^{-k}\epsilon)}_{\bullet}||)\arrow["f_*"]{r}& H_{p+q}(||\widetilde{\Sigma}_{\bullet}^nF(D^n,\partial)||).
\end{tikzcd}
\end{gathered}
\end{equation}
So to prove $f_*$ is an isomorphism, we need to show that $f_p$ induces a homology isomorphism. Note that we have the following commutative diagram 
\begin{equation}\label{eq:4}
\begin{gathered}
\begin{tikzpicture}[node distance=4cm, auto]
  \node (A) {${CF_1(D^n, 2^{-k}\epsilon)}_{p}$};
  \node (B) [right of=A] {$\widetilde{\Sigma}_p^nF(D^n,\partial)$};
  \node (C) [below of=A, node distance=1.7cm] {$\Delta_{p}(D^n,k)$};  
  \node (D) [below of=B, node distance=1.7cm] {$\Delta_{p}(D^n, k),$};
  \draw[->] (C) to node {$=$} (D);
  \draw [->] (A) to node {$\tau$}(C);
  \draw [->] (A) to node {$f_p$} (B);
  \draw [->] (B) to node {$\pi$} (D);
\end{tikzpicture}
\end{gathered}
\end{equation}
where $\pi$ and $\tau$ are natural projection to the simplicial set $\Delta_{\bullet}(D^n,k)$. Hence to show that $f_p$ induces a homology isomorphism, it is enough to prove that that $f_p$ induces a homology isomorphism on the fibers of $\tau$ and $\pi$. 

We have three cases: \begin{itemize}
\item If $\beta=(B_0,\dots, B_p)$ lies in $\Delta_p(D^n,k)$ but not in $D^n_p(k)$, then the fiber of $\tau$ consists of those elements in $F(D^n)$ that have empty support which is a contractible space by \Cref{condi}. The fiber of $\pi$ over $\beta$ is a point. 
\item  If $\beta=(B_0,\dots, B_p)$ lies in $D_p^n(k)$ but not in $\partial D^n_p(k)$, then the fiber of $\tau$ over $\beta$ is the subspace $F_c(\cap_jB_j)$. The fiber of $\pi$ over $\beta$ is $F(D^n,\partial)$. Note that by the second condition in \Cref{condi}, the inclusion $F_c(\cap_jB_j)\hookrightarrow F(D^n,\partial)$ is a homology isomorphism. 
\item   If $\beta=(B_0,\dots, B_p)$ lies in $\partial D_p^n(k)$, then the fiber of $\tau$ over $\beta$ is acyclic by the third condition of \Cref{condi} and the fiber of $\pi$ over $\beta$ is a point. Therefore, $f_p$ induces a homology isomorphism which in turn implies that $f$ induces a homology isomorphism. 
\end{itemize}

Since ${CF_k(M,\epsilon)}_{\bullet}={CF_1(M,2^{-k}\epsilon)}^k$. Therefore, the fact that $f$ induces a homology isomorphism implies that the map 
\[
||{CF_k(D^n)}_{\bullet}||\to (\widetilde{\Sigma}^nF(D^n,\partial))^k.
\]
is also a homology isomorphism. Similar to the diagram \ref{eq:4}, one can fiber the map 
\[
D{F^f_k(D^n)}_{\bullet}\to (\pi^k)^{-1}(T_{k,\bullet}(D^n)),
\]
over $T_{k,\bullet}(D^n)$, to prove that 
\[
||D{F^f_k(D^n)}_{\bullet}||\to (\pi^k)^{-1}(||T_{k,\bullet}(D^n)||).
\]
is also a homology isomorphism. 
\end{proof}
\begin{rem}\label{disks}Let $U$ be an open subset of $M$ that is homeomorphic to the disjoint union of Euclidean spaces of dimension $n$. The same proof as the case of $D^n$ implies that
\[
F_c(U)\to F^f_c(U)\simeq \prod_{\pi_0(U)}\Omega^nF(D^n,\partial),
\]
is a homology isomorphism. 
\end{rem}
\subsubsection{Proof of \Cref{h}}\label{global} Since both $F$ and $F^f$ satisfy fragmentation property, the spaces $F_c(M)\simeq F_{\epsilon}(M)$ and $F_c^f(M)\simeq F^f_{\epsilon}(M)$ can be filtered and the natural map $F_{\epsilon}(M)\to F_{\epsilon}^f(M)$ respects the filtration. Hence,  it is enough to show that the induced map between filtration quotients induces a homology isomorphism. Using \Cref{filtrationquotients}, it is enough to prove that the induced map between pairs
\[
(F_k(M), DF_k(M))\to (F^f_k(M), DF^f_k(M)),
\]
induces a homology isomorphism.  Let us first show that $F_k(M)\to F^f_k(M)$ induces a homology isomorphism using the same idea as in the proof of \Cref{sss}. We use \Cref{resolution} to resolve $F_k(M)$ and  $F^f_k(M)$ by $F_k(M)_{\bullet}$ and $F^f_k(M)_{\bullet}$. Recall that $F_k(M, \epsilon)= (F_1(M, 2^{-k}\epsilon))^k$ and $F^f_k(M, \epsilon)= (F^f_1(M, 2^{-k}\epsilon))^k$. Therefore, it is enough to show that 
\[
F_1(M)_p\to F^f_1(M)_p,
\]
induces a homology isomorphism for each simplicial degree $p$.  To do so, we consider the commutative diagram
\begin{equation}\label{eq:6}
\begin{gathered}
\begin{tikzpicture}[node distance=4cm, auto]
  \node (A) {${F_1(M)}_{p}$};
  \node (B) [right of=A] {$F^f_1(M)_p$};
  \node (C) [below of=A, node distance=1.7cm] {$\Delta_{p}(M)$};  
  \node (D) [below of=B, node distance=1.7cm] {$\Delta_{p}(M),$};
  \draw[->] (C) to node {$=$} (D);
  \draw [->] (A) to node {$\tau$}(C);
  \draw [->] (A) to node {$f_p$} (B);
  \draw [->] (B) to node {$\pi$} (D);
\end{tikzpicture}
\end{gathered}
\end{equation}
where $\pi$ and $\tau$ are natural projection to the simplicial set $\Delta_{\bullet}(M)$. Hence to show that $f_p$ induces a homology isomorphism, it is enough to prove that that $f_p$ induces a homology isomorphism on the fibers of $\tau$ and $\pi$. Let $\sigma_p=(B_0, B_1, \dots, B_p)$ be a $p$-simplex in $\Delta_{p}(M)$. There are two cases for the fibers of $\tau$ and $\pi$ over $\sigma_p$:
\begin{itemize}
\item The intersection of $B_i$'s is empty. Therefore, the preimages of $\sigma_p$ under $\tau$ and $\pi$ are contractible by the first condition in \Cref{condi}.
\item The intersection of $B_i$'s is not empty. Given that the disks $B_i$'s are geodesically convex, their intersection is homeomorphic to a disk. Hence, the induced map on fibers over $\sigma_p$ is
\[
F_c(\cap_i B_i)\to F^f_c(\cap_i B_i),
\]
which is a homology isomorphism as we proved the main theorem for the disks relative to their boundaries in section \ref{diskcase}.
\end{itemize}
Similarly, we could see that $DF_k(M))\to DF^f_k(M)$ induces a homology isomorphism by fibering the semisimplicial resolutions $DF_1(M)_{\bullet}$ and $DF^f_1(M)_{\bullet}$ over the semisimplicial set $T_{1,\bullet}(M)$. Hence, we conclude that the map between pairs
\[
(F_k(M), DF_k(M))\to (F^f_k(M), DF^f_k(M)),
\]
induces a homology isomorphism.\qed

\begin{rem} One could give a different proof of \Cref{h} using \Cref{condi} and goodness of $F$ directly. As in Appendix \ref{sec2}, we could use the notion of lax support to show that 
  \[
 \underset{U\in{\bf D}(M)}{\textsf{hocolim }}F^f_c(U)\xrightarrow{\simeq} F^f_c(M).
 \]
 Since $F$ satisfies \Cref{condi}, the same proof implies that 
 \[
 \underset{U\in{\bf D}(M)}{\textsf{hocolim }}F_c(U)\xrightarrow{\simeq} F_c(M).
 \]

 Using \Cref{disks} for $U\in{\bf D}(M)$, we know that $F_c(U)\to F_c^f(U)$ is a homology isomorphism. Using the spectral sequence to compute the homology of the homotopy colimits and the bar construction model for the homotopy colimits, it is enough to prove that natural maps
 \[
 B_{\bullet}(F_c(-), {\bf D}(M), *)\to  B_{\bullet}(F^f_c(-), {\bf D}(M), *),
 \]
 is a homology isomorphism which easily follows from $F_c(U)\to F^f_c(U)$ being homology isomorphism for all $U\in {\bf D}(M)$. 
 \end{rem}
 \section{Mather-Thurston's theory for new transverse structures}\label{newMT} In this section,  we prove Mather-Thurston's type \Cref{MT} for foliated bundles with new transverse structures. We shall first explain in more detail, what it means for the functors $\textnormal{\text{Fol}}_c(M, \alpha)$ and $\textnormal{\text{Fol}}^{\textnormal{\text{PL}}}_c(M)$ to satisfy the c-principle.We then explain how Thurston avoids the local statement for the foliated bundle by using the method of the proof of \Cref{main}.
 \begin{itemize}
\item {\boldmath $\textnormal{\textbf{Fol}}_c(M, \alpha)$}: Let $(M, \alpha)$ be a contact manifold where $M$ is a manifold of dimension $2n+1$ and $\alpha$ is a smooth $1$-form such that $\alpha \wedge (d\alpha)^n$ is a volume form. The group of $C^{\infty}$-contactomorphisms consists of $C^{\infty}$-diffeomorphisms such that $f^*(\alpha)=\lambda_f \alpha$ where $\lambda_f$ is a non-vanishing smooth function on $M$ depending on $f$. Since we are working with orientation-preserving automorphisms, we assume that $\lambda_f$ is a positive function. Let $\Cont_c(M,\alpha)$ denote the group of compactly supported contactomorphisms with induced topology from $C^{\infty}$-diffeomorphisms. It is known that this group is also locally contractible (\cite{MR2509723}). Let $\dCont_c(M,\alpha)$ denote the same group with the discrete topology. 

The functor $\textnormal{\text{Fol}}_c(M, \alpha)$ is homotopy equivalent to $\overline{\BCont_c(M,\alpha)}$ which is the homotopy fiber of the natural map
\[
\BdCont_c(M,\alpha)\to \BCont_c(M, \alpha). 
\]
The space of formal sections in this case is easier to describe.

 Let $\Gamma_{2n+1,ct}$ be the etale groupoid whose space of objects is $\bR^{2n+1}$ and the space of morphisms is given by the germ of contactomorphisms of $(\bR^{2n+1}, \alpha_{st})$ where $\alpha_{st}$ is the standard contact form $dx_0+\sum_{i=1}^n x_{i+n}dx_i$. Note that the subgroup of $\mathrm{GL}_{2n+1}(\bR)$, formed by orientation preserving linear transformations that preserve $\alpha_{st}$ has $U_n$ as a deformation retract. Hence, the derivative of morphisms in $\Gamma_{2n+1, ct}$ at their sources induces the map
\[
\nu\colon\mathrm{B}\Gamma_{2n+1,ct}\to \mathrm{BU}_n.
\]
Let $\tau_M$ be the map $M\to \mathrm{BU}_n$ that classifies the tangent for the contact manifold $(M,\alpha)$. The space of formal sections, $\textnormal{\text{Fol}}^f_c(M, \alpha)$, is the space of lifts of the map $\tau_M$ to $\mathrm{B}\Gamma_{2n+1,ct}$
 \begin{equation*}\label{eq8}
\begin{gathered}
 \begin{tikzpicture}[node distance=2cm, auto]
  \node (C)  {$M$};  
  \node (D) [right of=C, node distance=2.5cm] {$ \mathrm{BU}_n.$};
    \node (B)  [above of=D, node distance=1.3 cm] {$\mathrm{B}\Gamma_{2n+1,ct}$};
  \draw[->] (C) to node {$\tau_M$} (D);
  \draw [->] (B) to node {$\nu$} (D);
    \draw [->, dashed] (C) to node {$$} (B);
\end{tikzpicture}
\end{gathered}
\end{equation*}
The universal foliated $M$-bundle $\overline{\BCont_c(M,\alpha)}\times M\to \overline{\BCont_c(M,\alpha)}$ with the transverse contact structure (i.e. the holonomy maps respect the contact structure of the fibers) induces a classifying map $\overline{\BCont_c(M,\alpha)}\times M\to \mathrm{B}\Gamma_{2n+1,ct}$. The adjoint of this classifying map induces a map $\textnormal{\text{Fol}}_c(M, \alpha)\to \textnormal{\text{Fol}}^f_c(M, \alpha)$. Rybicki  \cite[Section 11]{MR2729009} mentioned that an analog of the Mather-Thurston theorem is not known for smooth contactomorphisms and he continues saying that ``it seems likely that such a version could be established, but a possible proof seems to be hard". We show that the above adjoint map induces homology isomorphisms in the compactly supported case. The non-compactly supported version remains open. In particular, it is unknown whether the map $\overline{\BCont(\bR^{2n+1},\alpha_{st})}\to \overline{\mathrm{B}\Gamma_{2n+1,ct}}$ induces a homology isomorphism where $\overline{\mathrm{B}\Gamma_{2n+1,ct}}$ is the homotopy fiber of the map $\nu$. The original Thurston's technique is useful to avoid such subtle local statements to get the compactly supported version. 
\item {\boldmath $\textnormal{\textbf{Fol}}^{\textnormal{\textbf{PL}}}_c(M)$}: Let $M$ be an $n$-dimensional manifold that has a PL structure. Let $\text{PL}_{\bullet}(M)$ be the simplicial group of PL homeomorphisms of $M$. The set of $k$-simplices, $\text{PL}_k(M)$, is the group of PL homeomorphisms $M \times \Delta^k$ that commute with projection to $\Delta^k$. The topological group, $\text{PL}(M)$, of PL homeomorphisms of $M$ is the geometric realization of $\text{PL}_{\bullet}(M)$. Hence, the $0$-simplices of $\text{PL}_{\bullet}(M)$ is $\text{PL}(M)^{\delta}$ which is the group of PL homeomorphisms of $M$ as a discrete group. Therefore, we have a map
\[
\mathrm{B}\text{PL}(M)^{\delta}\to \mathrm{B}\text{PL}(M),
\]
whose homotopy fiber is denoted by $\overline{\mathrm{B}\text{PL}(M)}$. This space is homotopy equivalent to $\textnormal{\text{Fol}}^{\textnormal{\text{PL}}}_c(M)$. The space of formal sections is defined similarly to the contact case. Let $\Gamma_n^{\text{PL}}$ denote the etale groupoid whose space of objects is $\bR^n$ and whose space morphisms are given by germs of PL homeomorphisms of $\bR^n$. Note that a germ of PL homeomorphism at its sources in $\bR^n$ uniquely extends to a PL homeomorphism of $\bR^n$. Hence, we obtain a map
\[
\mathrm{B}\Gamma_n^{\text{PL}}\to \mathrm{B}\text{PL}(\bR^n).
\]
Let $\tau_M\colon M\to \mathrm{B}\text{PL}(\bR^n)$ be a map that classifies the tangent microbundle of $M$. The space $\textnormal{\text{Fol}}^{\textnormal{\text{f,PL}}}_c(M)$ is the space of lifts of $\tau_M$ in the diagram
\begin{equation*}\label{eq8}
\begin{gathered}
 \begin{tikzpicture}[node distance=2cm, auto]
  \node (C)  {$M$};  
  \node (D) [right of=C, node distance=2.8cm] {$\mathrm{B}\text{PL}(\bR^n).$};
    \node (B)  [above of=D, node distance=1.3 cm] {$\mathrm{B}\Gamma_n^{\text{PL}}$};
  \draw[->] (C) to node {$\tau_M$} (D);
  \draw [->] (B) to node {$\nu$} (D);
    \draw [->, dashed] (C) to node {$$} (B);
\end{tikzpicture}
\end{gathered}
\end{equation*}
Similar to the previous cases, the universal foliated $M$ bundle with transverse PL structure induces a map $\overline{\mathrm{B}\text{PL}(M)}\times M\to \mathrm{B}\Gamma_n^{\text{PL}}$ whose adjoint gives the map $\textnormal{\text{Fol}}^{\textnormal{\text{PL}}}_c(M)\to \textnormal{\text{Fol}}^{\textnormal{\text{f,PL}}}_c(M)$. We show that this map induces a homology isomorphism which answers a question of Gelfand and Fuks in \cite[Section 5]{MR0339195}. However, the non-compactly supported version even for $M=\bR^n$ is not known. 
 \end{itemize}
\subsection{Strategy to avoid the local data} Recall that the strategy is first to prove $F(D^n, \partial)\to \Omega^n F^f(D^n)$ induces a homology isomorphism and then for a compact manifold $M$ the proof is exactly the same as \Cref{global}.  For smooth foliations without extra transverse structures, Thurston's idea to avoid the local statement as is explained in Mather's note (\cite{mather2011homology}) is to consider {\it disk model} for $\overline{\mathrm{B}\Gamma}_n$ (see \cite[Section 9]{mather2011homology}). More concretely, he proved that $\overline{\BDiff_c(\bR^n)}\to \Omega^n\overline{\mathrm{B}\Gamma}_n$ induces a homology isomorphism without using $\overline{\BDiff(\bR^n)}\to \overline{\mathrm{B}\Gamma}_n$ being a homology isomorphism. 

To recall the disk model for  $\overline{\mathrm{B}\Gamma}_n$, we define $F(D^n)$ to be the realization of the semi-simplicial set whose $q$-simplices are given by germs of foliations on the total space of $\Delta^q\times \bR^n\to \Delta^q$ around $\Delta^q\times D^n$ that are transverse to the fibers. It is easy to show that $F(D^n)$ is homotopy equivalent to $\overline{\mathrm{B}\Gamma}_n$. The advantage of the disk model is one can define the support for the germ of the foliation by taking the intersection of the support of a representative with the disk $D^n$ and it has Thurston's fragmentation property. But note that if a germ of a foliation is supported in an open set $U$ in $\text{int}(D^n)$, it would give a simplex in $\textnormal{\text{Fol}}_c(U)\simeq \overline{\BDiff_c(U)}$. In particular, we have $F(D^n, \partial)\simeq \overline{\BDiff_c(\bR^n)}$.

Similarly, we define $F^f(D^n)$ to be the space of maps $\text{Map}(D^n,\overline{\mathrm{B}\Gamma}_n)$. Since $\overline{\mathrm{B}\Gamma}_n$ is at least $(n-1)$-connected, it has the fragmentation property and given that $D^n$ is contractible, we have $F(D^n)\simeq F^f(D^n)$. Also we have $F^f(D^n, \partial)\simeq \Omega^n \overline{\mathrm{B}\Gamma}_n$. Therefore, we have a diagram \ref{R} for these choices and \Cref{j_1} applies to prove that $\overline{\BDiff_c(\bR^n)}\to \Omega^n\overline{\mathrm{B}\Gamma}_n$ induces a homology isomorphism.

One can use the corresponding disk model for each case in \Cref{MT}, and follow the same strategy as \Cref{main}. Hence, in each case, to show that the corresponding $F$ satisfies the c-principle,  we need to show that the fragmentation properties and the goodness conditions (\Cref{condi}) for $F$ are satisfied. 
 
 It is easy to see that these functors satisfy the first and the fourth conditions in \Cref{condi}. Since the subspace of foliations with empty support is a point, hence contractible. And the third and fourth conditions are obvious in this case. The second condition is also satisfied for these spaces of foliations because there exists a metric on the space of foliations that makes them complete metric spaces (see \cite[Section 2]{hirsch1973stability}). \footnote{Epstein (\cite[Section 6]{epstein1977topology} showed that in the case of smooth foliations, the topology induced by such metric is the same as the subspace topology of space of plane fields.} Hence, it is easy to see that the base point in these spaces which is the horizontal foliation makes them well-pointed, in particular, it is a strong neighborhood deformation retract. Therefore, similar to \Cref{lax}, all these functors satisfy \Cref{condi} meaning that enlarging the subspace of compactly supported foliations to lax compactly supported (which is an open subspace) does not change the homotopy type. Hence, to prove the goodness of these functors we need to check the last two conditions in \Cref{condi}.

 The case of the contactomorphisms and PL homeomorphisms are similar and given what we already know about the connectivity of the corresponding Haefliger spaces, as we shall see, we have all the ingredients to check the above conditions. Hence, we prove the c-principle for $\text{Fol}_c(M,\alpha)$ and $\text{Fol}_c^{\text{PL}}(M)$ first. 
 \subsection{The case of the contactomorphisms and PL homeomorphisms} Haefliger's argument in \cite[Section 6]{haefliger1971homotopy} implies that $\overline{\mathrm{B}\Gamma_n^{\text{PL}}}$ is $(n-1)$-connected. Haefliger showed (\cite[Theorem 3]{haefliger1970feuilletages}) that Phillips' submersion theorem in the smooth category implies that $\overline{\mathrm{B}\Gamma_n}$ is $n$-connected. Given that Phillips' submersion theorem also holds in the PL category (\cite{haefliger1964classification}), one could argue exactly similar to the smooth case that $\overline{\mathrm{B}\Gamma_n^{\text{PL}}}$ is, in fact, $n$-connected. On the other hand, McDuff in \cite[Proposition 7.4]{mcduff1987applications} also proved that $\overline{\mathrm{B}\Gamma_{2n+1,ct}}$ is $(2n+1)$-connected which is even one degree higher than what we need. Hence, $\text{Fol}^{f,\text{PL}}(D^n)$ is $n$-connected and $ \text{Fol}^f(D^{2n+1},\alpha)$ is $(2n+1)$-connected. Therefore, the space of formal sections satisfies the fragmentation property. To prove the fragmentation property for $\text{Fol}^{\text{PL}}(D^n)$ and $ \text{Fol}(D^{2n+1},\alpha)$, we shall use the following lemma.
 
 \begin{lem}[McDuff ]\label{mcduff} Let $G(M)$ be a topological group of compactly supported automorphisms of $M$ with a transverse geometric structure (e.g. PL homeomorphisms, contactomorphisms, volume-preserving diffeomorphisms, and foliation preserving diffeomorphisms). We assume that 
 \begin{itemize}
\item $G(M)$ with its given topology is locally contractible. 
\item For every isotopy $h_t$ as a path in $G(M)$ and every open cover $\{B_i\}_{i=1}^k$ of $M$, we can write $h_t=h_{t,1}\circ\cdots h_{t,k}$ where each $h_{t,i}$ is an isotopy supported in $B_i$. 
 \end{itemize}
Let $F_c(M)$ be the realization of the semisimplicial set whose $p$ simplices are the set of foliations on $\Delta^p\times M$ transverse to fibers of the projection $\Delta^p\times M\to \Delta^p$ and whose holonomies lie in $G(M)$. Then the functor $F_c(M)$ which is homotopy equivalent to $\overline{\mathrm{B}G(M)}$ has the fragmentation property in the sense of \Cref{prop}.
 \end{lem}
 \begin{proof}
 See section 4 of \cite{MR699012} and the discussion in subsection 4.15.
 \end{proof}
 The PL homeomorphism groups are known to be locally contractible (\cite{gauld1976local}) and as is proved by Hudson (\cite[Theorem 6.2]{hudson1969piecewise}) they also satisfy an appropriate isotopy extension theorem which gives the second condition in \Cref{mcduff}. Therefore, $\text{Fol}_c^{\text{PL}}(M)$ satisfies the fragmentation property. On the other hand, the group of contactomorphisms is also locally contractible (\cite[Section 3]{MR2509723}) and it satisfies the second condition (\cite[Lemma 5.2]{MR2729009}). Hence, $\text{Fol}_c(M,\alpha)$ also satisfies the fragmentation property in the sense of \Cref{prop}.
 
 Now we are left to show that $\textnormal{\text{Fol}}_c(-,\alpha)$ and $\textnormal{\text{Fol}}_c^{\text{PL}}(-)$ are good functors in the sense of \Cref{condi}. 
 \begin{lem}\label{goodness}
  $\textnormal{\text{Fol}}_c(-,\alpha)$ and $\textnormal{\text{Fol}}_c^{\text{PL}}(-)$ are good functors.
 \end{lem}
 \begin{proof}
 Recall that we need to check the last two conditions in \Cref{condi}. We focus on $\text{Fol}_c(-,\alpha)$ and we mention where $\text{Fol}_c^{\text{PL}}(-)$ is different.
 
 We may assume that  $U$ and $V$ are balls $B(r)$ and $B(R)$ of radi $r<R$ in $\bR^{2n+1}$. And we want to show that the induced map
 \[
\iota\colon\overline{\BCont_c(B(r),\alpha_{st})}\to \overline{\BCont_c(B(R),\alpha_{st})},
 \]
 is a homology isomorphism. 
 
 Note that for any topological group $G$, the homology of $\overline{\mathrm{B}G}$ can be computed by subchain complex $\text{S}_{\bullet}(\overline{\mathrm{B}G})$ of singular chains $\text{Sing}_{\bullet}(G)$ of the group $G$ given by smooth chains $\Delta^{\bullet}\to G$ that sends the first vertex to the identity (see section 1.4 of \cite{haller1998perfectness} for more detail). Given a chain $c$ in $ \text{S}_{\bullet}(\overline{\BCont_c(B(R),\alpha_{st})})$, to find a chain homotopy to a chain in $\overline{\BCont_c(B(r),\alpha_{st})}$, we need an easy lemma (\cite[Lemma 1.4.8]{haller1998perfectness}) which says that for every contactomorphism $h\in \Cont_{c,0}(B(R),\alpha_{st})$ that is isotopic to the identity, the conjugation by $h$ induces a self-map of $ \overline{\BCont_c(B(R),\alpha_{st})}$ which is the identity on homology. Hence, it is enough to show that there exists $h$ a contactomorphism, isotopic to the identity, that shrinks the support of the given chain to lie inside $B(r)$.  
 
 To find such compactly supported contraction, consider the following family of contactomorphisms $\rho_t: \bR^{2n+1}\to \bR^{2n+1} $
 \[
 \rho(x_0,x_1,\cdots, x_{2n+1})=(t^2.x_0,t.x_1,\cdots, t.x_{2n+1}).
 \]
 For $t<1$, it is a contracting contactomorphism but it is not compactly supported. To cut it off, we use the fact that the family $\rho_t$ is generated by a vector field $\dot{\rho}_t$. 
 
 Let $\lambda$ be a bump function that is positive on the support of the chain $c$ and zero near the boundary of $B(R)$. One wants to consider the flow of the vector field $\dot{\rho}_{\lambda.t}$ to cut off $\rho_t$ but $\dot{\rho}_{\lambda.t}$ may not be a contact vector field. However, there is a retraction $\pi$ from the Lie algebra of smooth vector fields to contact vector fields on every contact manifold (\cite[Section 1.4]{MR1445290}).  Briefly, the reason that this retraction exists is that there is an isomorphism (see \cite[Proposition 1.3.11]{MR1445290}) between contact vector fields on a contact manifold $(M, \alpha)$ and smooth functions by sending a contact vector field $\xi$ to $\iota_{\xi}(\alpha)$, the contraction of $\alpha$ by $\xi$. The retraction $\pi$ is defined by sending a smooth vector field to the contact vector field associated with the function $\iota_{\xi}(\alpha)$. Therefore, the flow of $\pi(\dot{\rho}_{\lambda.t})$ gives a family of compactly supported contactomorphisms of $B(R)$ that shrinks the support of the chain $c$. 
 
 Hence, by conjugating by such contactomorphisms that are isotopic to the identity, we conclude that $\iota$ induces a homology isomorphism. The case of the PL foliations is much easier because the existence of such contracting PL homeomorphisms that are isotopic to the identity is obvious. 
 
To check the last condition, we want to show that any chain $c$ in $S_{\bullet}(\overline{\BCont(D^{2n+1}, \partial_1)})$ is chain homotopic to the identity. Recall that the chain complex $S_{\bullet}(\overline{\BCont(D^{2n+1}, \partial_1)})$ is generated by the set of smooth maps from $\Delta^{\bullet}$ to $\Cont(D^{2n+1}, \partial_1)$ that send the first vertex to the identity contactomorphism. Note that this set is in bijection with the set of foliations on the total space of the projection $\Delta^{\bullet}\times M\to \Delta^{\bullet}$ that are transverse to the fibers $M$ and whose holonomies lie in $\Cont(D^{2n+1}, \partial_1)$. Thus, it is enough to show that each of these generators is chain homotopically trivial. Geometrically, this means that for each such foliation $c$ on $\Delta^{\bullet}\times M$, there is a foliation on $\Delta^{\bullet}\times [0,1]\times M$ transverse to the projection to the first two factors (i.e. it is a concordance) such that on $\Delta^{\bullet}\times \{0\}\times M$, it is given by $c$ and on $\Delta^{\bullet}\times \{1\}\times M$, it is the horizontal foliation. The idea is to ``push" the support of the foliation $c$ towards the free boundary until the foliation becomes completely horizontal.

To do so, consider a small neighborhood $U$ of $\partial_1$ in $D^{2n+1}$ that is in the complement of the support of the foliation $c$. Note that as in the previous case, there is a contact contraction that maps $D^{2n+1}$ to $U$ and is isotopic to the identity. Let us denote this contact isotopy by $h_t$ such that $h_0=\text{id}$. And let $F\colon \Delta^{\bullet}\times [0,1]\times D^{2n+1}\to \Delta^{\bullet}\times D^{2n+1}$ be the map that sends $(s,t,x)$ to $(s, h_t(x))$ and $F_t$ be the map $F$ at time $t$. Since $h_t$ is a contact isotopy, for each $t$, the pullback foliation $F_t^*(c)$ on $ \Delta^{\bullet}\times D^{2n+1}$ also gives an element in $S_{\bullet}(\overline{\BCont(D^{2n+1}\partial_1)})$. Therefore, the pullback foliation $F^*(c)$ on  $\Delta^{\bullet}\times [0,1]\times D^{2n+1}$ is a concordance from $c$ to the horizontal foliation which means that $c$ is chain homotopic to the identity in the chain complex $S_{\bullet}(\overline{\BCont(D^{2n+1}, \partial_1)})$. 

Note that we only used that for each foliation $c$ on $\Delta^{\bullet}\times D^{2n+1}$, there exists a neighborhood $U$ away from the support of $c$ and there is a contact embedding $h$ that maps $D^{2n+1}$ into $U$ which is also isotopic to the identity. Such embeddings isotopic to the identity also exist in the PL case. Therefore, $\text{Fol}_c(-,\alpha)$ and $\text{Fol}_c^{\text{PL}}(-)$ both satisfy the  conditions of \Cref{condi}.  
 \end{proof}
 As an application of this theorem, we could improve the connectivity of $\overline{\mathrm{B}\Gamma_{2n+1,ct}}$.
 \begin{cor}\label{improve}
 The classifying space $\overline{\mathrm{B}\Gamma_{2n+1,ct}}$ is at least $(2n+2)$-connected.
 \end{cor}
 \begin{proof} We already know by McDuff's theorem(\cite[Proposition 7.4]{mcduff1987applications}) that $\overline{\mathrm{B}\Gamma_{2n+1,ct}}$ is $(2n+1)$-connected. To improve the connectivity by one, note that Rybicki proved (\cite{MR2729009}) that $\overline{\BCont_c(\bR^{2n+1})}$ has a perfect fundamental group. Therefore, its first homology vanishes. On the other hand, by \Cref{MT}, the space $\overline{\BCont_c(\bR^{2n+1})}$ is homology isomorphic to $\Omega^{2n+1}\overline{\mathrm{B}\Gamma_{2n+1,ct}}$. Hence, we have 
 \[
 0=H_1(\Omega^{2n+1}\overline{\mathrm{B}\Gamma_{2n+1,ct}};\bZ)=\pi_1(\Omega^{2n+1}\overline{\mathrm{B}\Gamma_{2n+1,ct}})=\pi_{2n+2}(\overline{\mathrm{B}\Gamma_{2n+1,ct}}),
 \]
 which shows that $\overline{\mathrm{B}\Gamma_{2n+1,ct}}$ is $(2n+2)$-connected.
 \end{proof}
 However, as we mentioned in the introduction, the perfectness of the identity component of PL homeomorphisms $\text{PL}_0(M)^{\delta}$ of a PL manifold $M$ is not known in general. Epstein (\cite{epstein1970simplicity}) proved that $\text{PL}_c(\bR)^{\delta}$ and $\text{PL}_0(S^1)^{\delta}$ are perfect and left the case of higher dimensions as a question. In \cite[Theorem 1.4]{nariman2022flat}, the author used the c-principle of $\text{Fol}_c^{\text{PL}}(-)$ to prove the following.
 \begin{thm}Let $\Sigma$ be an oriented surface so it has essentially a unique PL structure and let $\textnormal{\text{PL}}_0(\Sigma,\text{rel }\partial)$ denote the identity component of the group of PL homeomorphisms of $\Sigma$ whose supports are away from the boundary. Then the discrete group ${\textnormal{\text{PL}}}^{\delta}_0(\Sigma,\text{rel }\partial)$ is a perfect group. 
 \end{thm}

 \section{Further discussion}
 \subsection{Space of functions not having certain singularities} It would be interesting to see if the space of smooth functions on $M$ not having certain singularities satisfies the fragmentation property.  In particular, it would give a different proof of Vassiliev's c-principle theorem (\cite[Section $3$]{MR1168473}) using the fragmentation method. Let $S$ be a closed semi-algebraic subset of the jet space $J^r(\bR^n;\bR)$ of codimension $n+2$ which is invariant under the lift of $\Diff(\bR^n)$ to the jet space. We denote the space of functions over $M$ avoiding the singularity set $S$ by $F(M,S)$. It is easy to check that  $F$ is a good functor satisfying the conditions \ref{condi}. To prove that $F$ satisfies the c-principle, we need to check whether the functors $F$ and $F^f$ satisfy the fragmentation property. It seems plausible to the author, as we shall explain, that using appropriate transversality argument for stratified manifolds ought to prove fragmentation property for $F(M)$ but since we still do not know if fragmenting the space of functions $F(M)$ is independently interesting, we do not pursue it further in this paper. 
 
 Recall, to check that the space of formal sections $F^f$ has fragmentation property, we have to show that $F(\bR^n,S)$ is at least $(n-1)$-connected. But it is easy to see that $F(\bR^n,S)$ is homotopy equivalent $J^r(\bR^n,\bR)\backslash S$ (see \cite[Lemma 5.13]{kupers2017three}) and this space by Thom's jet transversality is at least even $n$-connected. Therefore, the space of formal sections $F^f$ has the fragmentation property. 
 
 It is still not clear to the author how to check whether $F$ has the fragmentation property but here is an idea inspired by the fragmentation property for foliations. We want to solve the following lifting problem up to homotopy.
  \[ 
\begin{tikzpicture}[node distance=2cm, auto]
  \node (A) {$P$};
  \node (B) [right of=A] {$F_{\epsilon}(M,S)$};
  \node (C) [below of=A, node distance=1.5cm] {$Q$};  
  \node (D) [below of=B, node distance=1.5cm] {$ F(M,S),$};
  \draw[->, dashed] (C) to node  {$$} (B);
  \draw[->] (C) to node {$g$} (D);
  \draw [right hook->] (A) to node {}(C);
  \draw [->] (A) to node {$$} (B);
  \draw [right hook->] (B) to node {$$} (D);
\end{tikzpicture}
\]
where $Q$ is a simplicial complex and $P$ is a subcomplex. Let $\sigma\subset Q$ be a simplex. We can think of the restriction of $g$ to each $\sigma$ by adjointness as a map $g\colon M\times \sigma\to \bR$. In \Cref{thurston}, we defined the fragmentation homotopy $H_1\colon M\times \sigma\to M\times \sigma$ after fixing a partition of unity $\{\mu_i\}_{i=1}^N$. We have the flexibility to choose this partition of unity. Note that for each point ${\bold t}\in \sigma$, the space $H_1(M\times\{{\bold t}\})$ is diffeomorphic to $M$ (see \Cref{frag}). So the restriction of the map $g$ to this space gives a smooth function on $M$. By jet transversality, we can choose the triangulation of $Q$ fine enough so that for each simplex $\sigma$ and each point ${\bold t}\in \sigma$, the restriction of $g$ to $H_1(M\times\{{\bold t}\})$ avoids the singularity type $S$. 

Let $f_0$ be a function in $F(M,S)$ that we fix as a base section to define the support of other functions with respect to $f_0$. Similar to the proof of \Cref{fragment}, consider the subcomplex $L(\sigma)$ which is an $n$-dimensional subcomplex of $M\times \sigma$ which is the union of finitely many manifolds $L_i$ that are canonically diffeomorphic to $M$. In fact $L(\sigma)$ is a union of the graphs of finitely many functions $M\to \sigma$ inside $M\times \sigma$. It is easy to choose the partition of unity so that $L(\sigma)$ is a stratified manifold. The goal is to find a homotopy of the family of functions in $F(M,S)$ denoted by $g_s\colon M\times \sigma\to \bR$ so that $g_0=g$, and $g_1$ restricted to $H_1(M\times\{{\bold t}\})$ for each ${\bold t}\in \sigma$ is in $F(M,S)$, and most importantly the restriction of $g_1$ to each $L_i$ is given by the base function $f_0$. If we can find such homotopy then the rest of the proof is similar to proving fragmentation property for space of sections in \Cref{fragment}.  
%
\subsection{Foliation preserving diffeomorphism groups} Another interesting transverse structure is the foliation preserving case when we have a flag of foliations. To explain the functor in this case, let $M$ be a smooth $n$-dimensional manifold and $\mathcal{F}$ be a codimension $q$ foliation on $M$. Let $\text{Fol}_c(M, \mathcal{F})$ be the realization of the simplicial set whose $k$-simplices are given by the set of codimension $\text{dim}(M)$ foliations on $M \times \Delta^k$ that are transverse to the fibers of the projection $M \times \Delta^k\to \Delta^k$ and the holonomies are compactly supported diffeomorphisms of the fiber $M$ that preserve the leaves of $\mathcal{F}$.\footnote{holonomies must be leaf preserving. One can also define a version that holonomies may not preserve the leaves but they preserve the foliation. By a recent result in \cite{meersseman2018automorphism}, this version does not satisfy the fragmentation property in general.}

 To define the space of formal sections in this case, note that the foliation $\mathcal{F}$ on $M$ gives a lifting of the classifying map for the tangent bundle of $M$ to $\mathrm{B}\Gamma_q\times \mathrm{BGL}_{n-q}(\bR)$ where $\Gamma_q$ is the Haefliger groupoid of germs of diffeomorphisms $\bR^q$. Now consider the diagram
 \begin{equation}\label{eq7}
\begin{gathered}
 \begin{tikzpicture}[node distance=2cm, auto]
  \node (C)  {$M$};  
  \node (D) [right of=C, node distance=3cm] {$\mathrm{B}\Gamma_q\times \mathrm{BGL}_{n-q}(\bR),$};
    \node (B)  [above of=D, node distance=1.3 cm] {$\mathrm{B}\Gamma_{n,q}$};

  \draw[->] (C) to node {$\tau_{\mathcal{F}}$} (D);
  \draw [->] (B) to node {$\theta$} (D);
\end{tikzpicture}
\end{gathered}
\end{equation}
where $\Gamma_{n,q}$ is the subgroupoid $\Gamma_n$ given by germs of diffeomorphisms of $\bR^n$ that preserve the standard codimension $q$ foliation on $\bR^n$ (see \cite[Section 1.1]{MR3477863} for more details). Let $\overline{\mathrm{B}\Gamma}_{n,q}$ denote the homotopy fiber of $\theta$. Let $\textnormal{\text{Fol}}^f_c(M, \mathcal{F})$ be the space of lifts of $\tau_{\mathcal{F}}$ to $\mathrm{B}\Gamma_{n,q}$ up to homotopy.\footnote{The classifying spaces in the diagram \ref{eq7} are defined up to homotopy but if we fix models for them so that $\theta$ is a Serre fibration, $\textnormal{\text{Fol}}^f_c(M, \mathcal{F})$ is homotopy equivalent to the space of lifts of $\tau_{\mathcal{F}}$ along $\theta$ in that model.} Since the trivial $M$-bundle $\text{Fol}_c(M, \mathcal{F})\times M$ is the universal trivial foliated $M$-bundle whose holonomy preserves the leaves of $\mathcal{F}$, we have a homotopy commutative diagram
 \begin{equation*}\label{eq8}
\begin{gathered}
 \begin{tikzpicture}[node distance=1.6cm, auto]
  \node (A) {$\text{Fol}_c(M, \mathcal{F})\times M$};
  \node (C)  [below of=A]{$M$};  
  \node (D) [right of=C, node distance=3cm] {$\mathrm{B}\Gamma_q\times \mathrm{BGL}_{n-q}(\bR).$};
    \node (B)  [right  of=A, node distance=3 cm] {$\mathrm{B}\Gamma_{n,q}$};

  \draw[->] (C) to node {$\tau_{\mathcal{F}}$} (D);
  \draw [->] (A) to node {$$} (B);
    \draw [->] (A) to node {$$} (C);
  \draw [->] (B) to node {$\theta$} (D);
\end{tikzpicture}
\end{gathered}
\end{equation*}
The adjoint of the top horizontal map induces the map $\text{Fol}_c(M, \mathcal{F}))\to \text{Fol}^f_c(M, \mathcal{F})$. The method of this paper can be applied to show that $\text{Fol}_c(M, \mathcal{F})$ also satisfies the c-principle but we pursue this direction and its consequences elsewhere.

\section{Appendix}
\label{sec2}

In this section, we prove \Cref{nonabelian} using  Thurston's fragmentation of section spaces. The non-abelian Poincar{\' e} duality has been proved by various methods (see \cite{lurie2016higher, segal1973configuration, MR0358766, MR922926, MR1851264, MR3431668}). What makes the fragmentation property more useful in the geometric context in particular in foliation theory is that it lets us deform certain spaces associated with a manifold (e.g. section spaces and spaces of foliation with certain transverse structures)  to its subspace (instead of a homotopy colimit) that has a natural filtration (e.g. it deforms the section space to those sections whose supports have a volume less than $\epsilon$). 

The non-abelian Poincar{\' e} duality holds for topological manifolds with the same statement. But we are assuming (\Cref{complete}) that $M$ admits a  metric for which there exists $\epsilon>0$ such that all balls of radius $\epsilon$ are geodesically convex. Therefore, we give proof using the fragmentation method under this assumption. 

Let us recall the setup again. We have a Serre fibration $\pi\colon E\to M$ over such manifold $M$. Let $s_0$ be a fixed section for $\pi$ and we call the base section. We assume this base section satisfies the condition in \Cref{goodsection}. 
\begin{cond*}\label{goodsection} There is a fiber preserving homotopy $h_t$ of $E$ such that $h_0=\text{id}$ and $h_1^{-1}(s_0(M))$ is a neighborhood $V$ of $s_0(M)$ in $E$ and $h_t(s_0(M))=s_0(M)$ for all $t$, in other words the base section is a {\it good} base point in the space of sections.  We assumed that $M$ is a geodesic space and there exists a positive $\epsilon$ so that every ball of radius $\epsilon$ is geodesically convex.
\end{cond*}
So with respect to this base section, we can define the support for any other section as in \Cref{complete}. Let $\text{Sect}_c(\pi)$ be the space of compactly supported sections of $\pi$ equipped with the compact-open topology. Let $\text{Sect}_{\epsilon}(\pi)$ denote the subspace of those sections $s$ such that the support of $s$ can be covered by $k$ geodesically convex balls of radius $2^{-k}\epsilon$ for some positive integer $k$. Recall that $\text{Disj}(M)$ is the poset of the open subsets of $M$ that are homeomorphic to a disjoint union of finitely many open disks. And for an open set $U\in \text{Disj}(M)$, the space $\text{Sect}_c(U)$ denotes the subspace of sections which are compactly supported and their supports are covered by $U$. The non-abelian Poincar{\' e} duality says that if the fiber of the map $\pi$ is $(n-1)$-connected, the natural map
\[
\underset{U\in\text{Disj}(M)}{\textsf{hocolim }} \text{\textnormal{Sect}}_c(U)\to \text{\textnormal{Sect}}_c(\pi),
\]
is a weak homotopy equivalence. 

To prove this statement, we shall recall below the reformulation due to Lurie  \cite[Proposition 5.5.2.13]{lurie2016higher} in terms of a more flexible indexing category ${\bf D}(M)$. And to use the fragmentation method, we shall first describe  $\text{Sect}_{\epsilon}(\pi)$ as a homotopy colimit over the category $\mathcal{O}_{\epsilon}(M)$. Recall from \Cref{O}, that this category is the discrete poset of open subsets of $M$ that can be covered by a union of $k$ geodesically convex balls of radius at most $2^{-k}\epsilon$ for some positive integer $k$.


Recall that by the fragmentation method (\Cref{fragment}), we know that the inclusion $\text{Sect}_{\epsilon}(\pi)\hookrightarrow \text{Sect}_{c}(\pi)$ is a weak homotopy equivalence. Hence, we want to compare $\text{Sect}_{\epsilon}(\pi)$ with $\underset{U\in\text{Disj}(M)}{\textsf{hocolim }} \text{\textnormal{Sect}}_c(U)$, and to do so, we shall define some auxiliary spaces.

\begin{defn}\label{lax}
We define the lax support of a section $s\in \text{Sect}_c(\pi)$ to be the closure of the set of points $x$ where $s(x)$ is not in the neighborhood $V$ that is chosen in the condition above. 
\end{defn}
\begin{defn}
Let $\widetilde{\text{Sect}}_{c}(\pi)$ be the subspace of space of sections whose lax support is compact and in general, for an open set $U$ in $M$, let $\widetilde{\text{Sect}}_{c}(U)$ denote the space whose lax support is compact and lies inside $U$. Also, let $\widetilde{\text{Sect}}_{\epsilon}(\pi)$ be the subspace of $\widetilde{\text{Sect}}_{c}(\pi)$ consisting of those sections whose lax support can be covered by  $k$ geodesically convex balls of radius $2^{-k}\epsilon$ for some positive integer $k$. And similarly, let  $\widetilde{\text{Sect}}_{\epsilon}(U)$ denote the subspace of $\widetilde{\text{Sect}}_{\epsilon}(\pi)$ consisting of those sections whose lax supports can also be covered by $U$. 
\end{defn}
\begin{lem}\label{lax}
For an open set $U$ in $M$, the inclusion $\text{\textnormal{Sect}}_{c}(U)\hookrightarrow \widetilde{\text{\textnormal{Sect}}}_{c}(U)$ is a weak homotopy equivalence and similarly, the inclusion $\text{\textnormal{Sect}}_{\epsilon}(U)\hookrightarrow \widetilde{\text{\textnormal{Sect}}}_{\epsilon}(U)$ is a weak homotopy equivalence. 
\end{lem}
\begin{proof}
The proof of both statements is the same so we shall do the first. We need to solve the following lifting problem
\[ 
\begin{tikzpicture}[node distance=2cm, auto]
  \node (A) {$S^k$};
  \node (B) [right of=A] {$\text{\textnormal{Sect}}_{c}(U)$};
  \node (C) [below of=A, node distance=1.5cm] {$ D^{k+1}$};  
  \node (D) [below of=B, node distance=1.5cm] {$ \widetilde{\text{\textnormal{Sect}}}_{c}(U).$};
  \draw[->, dashed] (C) to node  {$$} (B);
  \draw[->] (C) to node {$H$} (D);
  \draw [right hook->] (A) to node {}(C);
  \draw [->] (A) to node {$F$} (B);
  \draw [->] (B) to node {$$} (D);
\end{tikzpicture}
\]
But instead we change the map of pairs $$(H,F): (D^{k+1}, S^k)\to (\widetilde{\text{\textnormal{Sect}}}_{c}(U), \text{\textnormal{Sect}}_{c}(U)),$$ up to homotopy to find the lift. For $a\in D^{k+1}$ and $x\in M$, we define $H_t(a,x)\in E$ to be $h_t(H(a,x))$. Similarly, we define $F_t$. Note that for all $a\in D^{k+1}$, the section $H_1(a,-)$ in fact lies in $\text{\textnormal{Sect}}_{c}(U)$ which is our desired lift.
\end{proof}

\begin{lem}\label{micro}
The natural map
\[
\eta\colon\underset{U\in\mathcal{O}_{\epsilon}(M)}{\textsf{hocolim }}  \widetilde{\text{\textnormal{Sect}}}_{\epsilon}(U)\to  \widetilde{\text{\textnormal{Sect}}}_{\epsilon}(\pi),
\]
is a weak homotopy equivalence. 
\end{lem}
\begin{proof}
Note that by definition the subspaces $\widetilde{\text{\textnormal{Sect}}}_{\epsilon}(U)$ are open in $\widetilde{\text{\textnormal{Sect}}}_{\epsilon}(\pi)$ and they give an open cover as $U$ varies in $\mathcal{O}_{\epsilon}(M)$. So the lemma is implied by \cite[Theorem 1.1]{MR2045835} but it is also easily implied by the microfibration technique as follows. It is enough to show that the above map is a Serre microfibration with weakly contractible fibers (see \cite[Lemma 2.2]{weiss2005does}).  Recall, we say the map $\pi: T\to B$ is a Serre microfibration if for every positive integer $k$ and every commutative diagram
\[ 
\begin{tikzpicture}[node distance=2cm, auto]
  \node (A) {$D^k\times \{0\}$};
  \node (B) [right of=A] {$T$};
  \node (C) [below of=A, node distance=1.4cm] {$ D^{k}\times [0,1]$};  
  \node (D) [below of=B, node distance=1.4cm] {$B,$};
  \draw[->] (C) to node {$h$} (D);
  \draw [right hook->] (A) to node {}(C);
  \draw [->] (A) to node {$f$} (B);
  \draw [->] (B) to node {$\pi$} (D);
\end{tikzpicture}
\]
there exists an $\epsilon>0$ and  a (micro)lift $H:D^k\times [0,\epsilon)\to T$ so that $H(x,0)=f(x)$ and $\pi\circ H(x,t)=h(x)$. 

We think of $\widetilde{\text{\textnormal{Sect}}}_{\epsilon}(-): \mathcal{O}_{\epsilon}(M)\to {\bf Top}$ as a diagram of spaces. It is known (see \cite[Appendix A]{MR2045835}) that for the diagram of spaces, the homotopy colimit is weakly equivalent to the realization of the bar construction $B_{\bullet}(\widetilde{\text{\textnormal{Sect}}}_{\epsilon}(-), \mathcal{O}_{\epsilon}(M), *)$. Note  that there is a continuous degree-wise injective map of semi-simplicial spaces
\[
B_{\bullet}(\widetilde{\text{\textnormal{Sect}}}_{\epsilon}(-), \mathcal{O}_{\epsilon}(M), *)\to \widetilde{\text{\textnormal{Sect}}}_{\epsilon}(\pi)\times B_{\bullet}(*, \mathcal{O}_{\epsilon}(M), *),
\]
where the map $\eta$ in the lemma is induced by the projection to the first factor. The lax support is defined so that the subspace $\widetilde{\text{\textnormal{Sect}}}_{\epsilon}(U)$ is open in $\widetilde{\text{\textnormal{Sect}}}_{\epsilon}(\pi)$ and since these spaces are Hausdorff, we could use \cite[Proposition 2.8]{galatius2018homological} to deduce that the map
\[
|B_{\bullet}(\widetilde{\text{\textnormal{Sect}}}_{\epsilon}(-), \mathcal{O}_{\epsilon}(M), *)|\to \widetilde{\text{\textnormal{Sect}}}_{\epsilon}(\pi),
\]
is a Serre microfibration. The fiber over a section $s\in \widetilde{\text{\textnormal{Sect}}}_{\epsilon}(\pi)$ can be identified with $|B_{\bullet}(*, \mathcal{O}_{\epsilon}(M)|_{\text{supp}(s)}, *)|$ where $\mathcal{O}_{\epsilon}(M)|_{\text{supp}(s)}$ consists of those open subsets in $\mathcal{O}_{\epsilon}(M)$ that contains the support of $s$. But this sub-poset is filtered, therefore its realization is contractible. 
\end{proof}
So using these spaces instead, we want to prove that 
\[
\underset{U\in\text{Disj}(M)}{\textsf{hocolim }} \widetilde{\text{\textnormal{Sect}}}_c(U)\to \widetilde{\text{\textnormal{Sect}}}_c(\pi),
\]
is a weak homotopy equivalence. However, the fragmentation method (\Cref{fragment}) implies that the inclusion $\widetilde{\text{\textnormal{Sect}}}_{\epsilon}(U)\hookrightarrow \widetilde{\text{\textnormal{Sect}}}_c(U)$ is a homotopy equivalence so we need to prove that the map 
\begin{equation}\label{w}
\begin{gathered}
\underset{U\in\text{Disj}(M)}{\textsf{hocolim }} \widetilde{\text{\textnormal{Sect}}}_{\epsilon}(U)\to \widetilde{\text{\textnormal{Sect}}}_c(\pi),
\end{gathered}
\end{equation}
is a weak homotopy equivalence.

On the other hand, combining \Cref{fragment} with \Cref{micro} and \Cref{lax}, we have the weak homotopy equivalence
\[
\underset{U\in\mathcal{O}_{\epsilon}(M)}{\textsf{hocolim }}  \widetilde{\text{\textnormal{Sect}}}_{\epsilon}(U)\xrightarrow{\simeq} \widetilde{\text{\textnormal{Sect}}}_{c}(\pi).
\]
Let us first observe that there is a functor $$\text{conv}\colon \mathcal{O}_{\epsilon}(M) \to \text{Disj}(M),$$ defined as follows. Recall that every open set $U$ in $\mathcal{O}_{\epsilon}(M)$ can be covered by a union of $k$ geodesically convex balls of radius at most $2^{-k}\epsilon$ for some positive integer $k$. 
\begin{lem} A union of $k$ geodesically convex balls of radius at most $2^{-k}\epsilon$ can be covered by at most $k$ {\it disjoint} geodesically convex balls of radius at most $\epsilon$. 
\end{lem} 
\begin{proof}Note that if the union of $r$ geodesically convex balls of radius at most $2^{-k}\epsilon$ is connected, it can be covered by a ball of radius at most $2^{-k+r-1}\epsilon$. So we consider the connected components of the union of $k$ balls of radius at most $2^{-k}\epsilon$ and we inductively cover the connected components by bigger balls if necessary until we obtain at most $k$ disjoint balls of radius at most $\epsilon$. 
\end{proof}
Let $\text{conv}(U)$ be the union of convex hulls of the connected components which is  homeomorphic to the disjoint union of balls in $M$. Hence, $\text{conv}(U)$ gives an object in $\text{Disj}(M)$ and it respects the containment so it is a functor between the two posets. Hence, we obtain a map
\[
\beta\colon\underset{U\in\mathcal{O}_{\epsilon}(M)}{\textsf{hocolim }}  \widetilde{\text{\textnormal{Sect}}}_{\epsilon}(U)\to \underset{U\in\text{Disj}(M)}{\textsf{hocolim }} \widetilde{\text{\textnormal{Sect}}}_{\epsilon}(U).
\]
Hence, to prove the non-abelian Poincar{\' e} duality for space of sections over $M$, it is enough to prove the following
\begin{thm}\label{colim}
The map $\beta$ induces a weak homotopy equivalence.
\end{thm}
In other words, we want to compare the homotopy colimit of two diagrams of section spaces over indexing categories $\mathcal{O}_{\epsilon}(M)\to \text{Disj}(M)$. To do that, we need a more flexible indexing category and the reformulation of the non-abelian Poincar{\' e} duality by Lurie (\cite[Theorem 5.5.6.6]{lurie2016higher}) in terms of this more flexible $\infty$-category.

\begin{defn}
Let $\mathsf{Mfld}_n$ denote the topological category of  $n$-dimensional topological manifolds and the morphisms are given by space of the codimension zero embeddings. We let ${\bf D}(M)$ to be the full subcategory of the $\infty$-category of $\text{\textnormal{N}}(\mathsf{Mfld}_n)_{/M}$ spanned
by those objects of the form $j\colon S\times \bR^n\to M$ where $S$ is a finite set. Here $\text{\textnormal{N}}(-)$ means the homotopy coherent nerve of the category (see \cite[ Definition 5.5.2.11]{lurie2016higher}). \end{defn} 
The space of morphisms $\text{Map}_{{\bf D}(M)}(f,g)$ between two objects  embeddings $(f: U\hookrightarrow M)$ and $(g: V\hookrightarrow M)$ in ${\bf D}(M)$ can be described by the following homotopy fiber sequence
\[
\text{Map}_{{\bf D}(M)}(f,g)\to \text{Sing}(\text{Emb}(U,V))\to \text{Sing}(\text{Emb}(U,M)),
\]
where the last map is induced by precomposing with $g$ and $\text{Sing}$ means the singular simplicial set. So roughly, we think of $\text{Map}_{{\bf D}(M)}(f,g)$ as the space of pairs of embeddings $(\iota, f)$ in $ \text{Emb}(U,V)$ and  $\text{Emb}(U,M)$ respectively and a specified isotopy in $\text{Emb}(U,M)$ between $f$ and  $g\circ\iota$.

Lurie in \cite[Remark 5.5.2.12]{lurie2016higher} defines an $\infty$-functor from the nerve of $\text{Disj}(M)$ to ${\bf D}(M)$
\[
\gamma\colon \text{N}(\text{Disj}(M))\to {\bf D}(M),
\]
by choosing a parametrization of each open disk in $M$. And in \cite[Proposition 5.5.2.13]{lurie2016higher}, he showed that the functor $\gamma$ is left cofinal. Hence, the colimits of diagrams over these $\infty$-categories are homotopy equivalent. The same argument as we shall sketch shows that the composition of functors
\[
\alpha\colon\text{N}(\mathcal{O}_{\epsilon}(M))\xrightarrow{\text{conv}}\text{N}(\text{Disj}(M))\xrightarrow{\gamma} {\bf D}(M),
\]
is also cofinal.
%
\begin{proof}
So we are left to show that $\alpha=\gamma\circ\text{conv}$ is also left cofinal similar to proposition \cite[Proposition 5.5.2.13]{lurie2016higher}. Let $V\in {\bf D}(M)$ and  ${\bf D}(M)_{V/}$ is the slice category under V. By Joyal's theorem \cite[Theorem 4.1.3.1]{MR2522659},  it is enough to show that $\text{N}(\mathcal{O}_{\epsilon}(M))\times_{{\bf D}(M)} {\bf D}(M)_{V/}$ is weakly contractible.  The projection $\text{N}(\mathcal{O}_{\epsilon}(M))\times_{{\bf D}(M)} {\bf D}(M)_{V/}\to \text{N}(\mathcal{O}_{\epsilon}(M))$ is a left fibration associated to a functor $\chi: \text{N}(O_\epsilon(M))\to {\bf Top}$ which sends $U\in \text{N}(\mathcal{O}_{\epsilon}(M))$ to the homotopy fiber of the map 
\[
\text{Sing}(\text{Emb}(V,\alpha(U)))\to \text{Sing}(\text{Emb}(V,M)).
\]
Hence, by \cite[Proposition 3.3.4.5]{MR2522659}, it is enough to show that 
\begin{equation}\label{eq:1}
\underset{ \text{N}(\mathcal{O}_{\epsilon}(M))}{\textsf{colim }}\text{Sing}(\text{Emb}(V,\alpha(-)))\xrightarrow{\simeq} \text{Sing}(\text{Emb}(V,M)),
\end{equation}
is a weak equivalence. Suppose that $V$ is homeomorphic to $S\times \bR^n$ for a finite set $S$. For any open subset $U$ in $M$, let $\text{Conf}(S,U)$ denote the space of embeddings of the set $S$ into $U$. Lurie showed (\cite[Remark 5.4.1.11]{lurie2016higher}) that the diagram
\[
 \begin{tikzpicture}[node distance=1.6cm, auto]
  \node (A) {$\text{Sing}(\text{Emb}(S\times \bR^n,U))$};
  \node (C)  [below of=A]{$\text{Sing}(\text{Conf}(S,U))$};  
  \node (D) [right of=C, node distance=5cm] {$\text{Sing}(\text{Conf}(S,M)),$};
    \node (B)  [right  of=A, node distance=5 cm] {$\text{Sing}(\text{Emb}(S\times \bR^n,M))$};

  \draw[->] (C) to node {$$} (D);
  \draw [->] (A) to node {$$} (B);
    \draw [->] (A) to node {$$} (C);
  \draw [->] (B) to node {$$} (D);
\end{tikzpicture}
\]
where the vertical maps are given by evaluation at $0$, is a homotopy cartesian diagram. Hence, the weak equivalence in \ref{eq:1} is equivalent to proving
\[
\underset{\text{N}(\mathcal{O}_{\epsilon}(M))}{\textsf{colim }}\text{Sing}(\text{Conf}(S, \alpha(-)))\xrightarrow{\simeq} \text{Sing}(\text{Conf}(S,M)).
\]
Note that $\text{Conf}(S,\alpha(U))$ is an open subspace $\text{Conf}(S,M)$ and as $U$ varies in $\mathcal{O}_{\epsilon}(M)$, the open subspaces $\text{Conf}(S,\alpha(U))$ cover $\text{Conf}(S,M)$. So by \cite[Theorem 1.1]{MR2045835}, the above map is a weak equivalence.
\end{proof}

\bibliographystyle{alpha}
\bibliography{reference}

\begin{thebibliography}{BdBW13}

\bibitem[AF15]{MR3431668}
David Ayala and John Francis.
\newblock Factorization homology of topological manifolds.
\newblock {\em J. Topol.}, 8(4):1045--1084, 2015.

\bibitem[B\"87]{MR922926}
C.-F. B\"{o}digheimer.
\newblock Stable splittings of mapping spaces.
\newblock In {\em Algebraic topology ({S}eattle, {W}ash., 1985)}, volume 1286
  of {\em Lecture Notes in Math.}, pages 174--187. Springer, Berlin, 1987.

\bibitem[Ban97]{MR1445290}
Augustin Banyaga.
\newblock {\em The structure of classical diffeomorphism groups}, volume 400 of
  {\em Mathematics and its Applications}.
\newblock Kluwer Academic Publishers Group, Dordrecht, 1997.

\bibitem[BdBW13]{MR3138384}
Pedro Boavida~de Brito and Michael Weiss.
\newblock Manifold calculus and homotopy sheaves.
\newblock {\em Homology Homotopy Appl.}, 15(2):361--383, 2013.

\bibitem[DI04]{MR2045835}
Daniel Dugger and Daniel~C. Isaksen.
\newblock Topological hypercovers and {$\Bbb A^1$}-realizations.
\newblock {\em Math. Z.}, 246(4):667--689, 2004.

\bibitem[EM02]{MR1909245}
Y.~Eliashberg and N.~Mishachev.
\newblock {\em Introduction to the {$h$}-principle}, volume~48 of {\em Graduate
  Studies in Mathematics}.
\newblock American Mathematical Society, Providence, RI, 2002.

\bibitem[Eps70]{epstein1970simplicity}
David~BA Epstein.
\newblock The simplicity of certain groups of homeomorphisms.
\newblock {\em Compositio Mathematica}, 22(2):165--173, 1970.

\bibitem[Eps77]{epstein1977topology}
DBA Epstein.
\newblock A topology for the space of foliations.
\newblock In {\em Geometry and topology}, pages 132--150. Springer, 1977.

\bibitem[ERW19]{ebert2017semi}
Johannes Ebert and Oscar Randal-Williams.
\newblock {Semisimplicial spaces}.
\newblock {\em Algebraic \& Geometric Topology}, 19(4):2099 -- 2150, 2019.

\bibitem[Fre20]{freedman2020controlled}
Michael Freedman.
\newblock Controlled {M}ather-{T}hurston theorems.
\newblock {\em arXiv preprint arXiv:2006.00374}, 2020.

\bibitem[Fuk74]{fuks1974quillenization}
DB~Fuks.
\newblock Quillenization and bordism.
\newblock {\em Functional Analysis and Its Applications}, 8(1):31--36, 1974.

\bibitem[Gau76]{gauld1976local}
David~B Gauld.
\newblock Local contractibility of spaces of homeomorphisms.
\newblock {\em Compositio Mathematica}, 32(1):3--11, 1976.

\bibitem[GF73]{MR0339195}
I.~M. Gelfand and D.~B. Fuks.
\newblock {${\rm PL}$} foliations.
\newblock {\em Funkcional. Anal. i Prilo\v{z}en.}, 7(4):29--37, 1973.

\bibitem[Gre92]{MR1200422}
Peter Greenberg.
\newblock Generators and relations in the classifying space for {PL}
  foliations.
\newblock {\em Topology Appl.}, 48(3):185--205, 1992.

\bibitem[GRW18]{galatius2018homological}
S{\o}ren Galatius and Oscar Randal-Williams.
\newblock Homological stability for moduli spaces of high dimensional
  manifolds. i.
\newblock {\em Journal of the American Mathematical Society}, 31(1):215--264,
  2018.

\bibitem[Hae70]{haefliger1970feuilletages}
Andr{\'e} Haefliger.
\newblock Feuilletages sur les vari{\'e}t{\'e}s ouvertes.
\newblock {\em Topology}, 9(2):183--194, 1970.

\bibitem[Hae71]{haefliger1971homotopy}
Andr{\'e} Haefliger.
\newblock Homotopy and integrability.
\newblock In {\em Manifolds-Amsterdam 1970}, pages 133--163. Springer, 1971.

\bibitem[Hal98]{haller1998perfectness}
Stefan Haller.
\newblock Perfectness and simplicity of certain groups of diffeomorphisms.
\newblock {\em PhD Thesis}, 1998.

\bibitem[Hir73]{hirsch1973stability}
Morris~W Hirsch.
\newblock Stability of compact leaves of foliations.
\newblock In {\em Dynamical systems}, pages 135--153. Elsevier, 1973.

\bibitem[HP64]{haefliger1964classification}
Andr{\'e} Haefliger and Valentin Poenaru.
\newblock La classification des immersions combinatories.
\newblock {\em Publications Math{\'e}matiques de l'Institut des Hautes
  {\'E}tudes Scientifiques}, 23(1):75--91, 1964.

\bibitem[Hud69]{hudson1969piecewise}
John~FP Hudson.
\newblock Piecewise linear topology.
\newblock {\em New York}, 1, 1969.

\bibitem[Igu84]{MR744854}
Kiyoshi Igusa.
\newblock On the homotopy type of the space of generalized {M}orse functions.
\newblock {\em Topology}, 23(2):245--256, 1984.

\bibitem[Igu87]{MR882699}
Kiyoshi Igusa.
\newblock The space of framed functions.
\newblock {\em Trans. Amer. Math. Soc.}, 301(2):431--477, 1987.

\bibitem[Knu18]{knudsen2018configuration}
Ben Knudsen.
\newblock Configuration spaces in algebraic topology.
\newblock {\em arXiv preprint arXiv:1803.11165}, 2018.

\bibitem[Kup19]{kupers2017three}
Alexander Kupers.
\newblock Three applications of delooping to h-principles.
\newblock {\em Geometriae Dedicata}, 202(1):103--151, 2019.

\bibitem[LM16]{MR3477863}
Fran\c{c}ois Laudenbach and Ga\"{e}l Meigniez.
\newblock Haefliger structures and symplectic/contact structures.
\newblock {\em J. \'{E}c. polytech. Math.}, 3:1--29, 2016.

\bibitem[Lur]{lurie2016higher}
Jacob Lurie.
\newblock Higher algebra, {S}eptember 2017.
\newblock {\em available at his webpage https://www. math. ias. edu/\~{}
  lurie}.

\bibitem[Lur09]{MR2522659}
Jacob Lurie.
\newblock {\em Higher topos theory}, volume 170 of {\em Annals of Mathematics
  Studies}.
\newblock Princeton University Press, Princeton, NJ, 2009.

\bibitem[Mat76]{mather2011homology}
John~N Mather.
\newblock On the homology of {H}aefliger's classifying space.
\newblock In {\em Differential Topology, V. Villani (Ed.). C.I.M.E. Summer
  Schools}, volume~73, pages 71--116. Springer. https://doi.org/10.1007/
  978-3-642-11102-0 4., 1976.

\bibitem[McD75]{MR0358766}
Dusa McDuff.
\newblock Configuration spaces of positive and negative particles.
\newblock {\em Topology}, 14:91--107, 1975.

\bibitem[McD79]{mcduff1979foliations}
Dusa McDuff.
\newblock Foliations and monoids of embeddings.
\newblock {\em Cantrell}, pages 429--444, 1979.

\bibitem[McD80]{mcduff1980homology}
Dusa McDuff.
\newblock The homology of some groups of diffeomorphisms.
\newblock {\em Commentarii Mathematici Helvetici}, 55(1):97--129, 1980.

\bibitem[McD81]{mcduff1981groups}
Dusa McDuff.
\newblock On groups of volume-preserving diffeomorphisms and foliations with
  transverse volume form.
\newblock {\em Proceedings of the London Mathematical Society}, 3(2):295--320,
  1981.

\bibitem[McD82]{MR707329}
Dusa McDuff.
\newblock Local homology of groups of volume preserving diffeomorphisms. {I}.
\newblock {\em Ann. Sci. \'Ecole Norm. Sup. (4)}, 15(4):609--648 (1983), 1982.

\bibitem[McD83a]{MR699012}
Dusa McDuff.
\newblock Local homology of groups of volume-preserving diffeomorphisms. {II}.
\newblock {\em Comment. Math. Helv.}, 58(1):135--165, 1983.

\bibitem[McD83b]{mcduff1983local}
Dusa McDuff.
\newblock Local homology of groups of volume-preserving diffeomorphisms. iii.
\newblock In {\em Annales scientifiques de l'{\'E}cole Normale Sup{\'e}rieure},
  volume~16, pages 529--540. Soci{\'e}t{\'e} math{\'e}matique de France, 1983.

\bibitem[McD87]{mcduff1987applications}
Dusa McDuff.
\newblock Applications of convex integration to symplectic and contact
  geometry.
\newblock In {\em Annales de l'institut Fourier}, volume~37, pages 107--133,
  1987.

\bibitem[Mei21]{meigniez2018quasi}
Ga{\"e}l Meigniez.
\newblock {Quasicomplementary foliations and the Mather–Thurston theorem}.
\newblock {\em Geometry \& Topology}, 25(2):643 -- 710, 2021.

\bibitem[Mil57]{MR0084138}
John Milnor.
\newblock The geometric realization of a semi-simplicial complex.
\newblock {\em Ann. of Math. (2)}, 65:357--362, 1957.

\bibitem[MNR18]{meersseman2018automorphism}
Laurent Meersseman, Marcel Nicolau, and Javier Ribon.
\newblock On the automorphism group of foliations with geometric transverse
  structure.
\newblock {\em arXiv preprint arXiv:1810.07244}, 2018.

\bibitem[MS76]{mcduff1976homology}
Dusa McDuff and Grame Segal.
\newblock Homology fibrations and the "group completion" theorem.
\newblock {\em Invent. Math.}, 31:279--284, 1976.

\bibitem[MW12]{MR2978449}
Scott Morrison and Kevin Walker.
\newblock Blob homology.
\newblock {\em Geom. Topol.}, 16(3):1481--1607, 2012.

\bibitem[Nar17]{nariman2014homologicalstability}
Sam Nariman.
\newblock Homological stability and stable moduli of flat manifold bundles.
\newblock {\em Advances in Mathematics}, 320:1227--1268, 2017.

\bibitem[Nar22]{nariman2022flat}
Sam Nariman.
\newblock On flat manifold bundles and the connectivity of {H}aefliger's
  classifying spaces.
\newblock {\em arXiv preprint arXiv:2202.00052}, 2022.

\bibitem[Ryb10]{MR2729009}
Tomasz Rybicki.
\newblock Commutators of contactomorphisms.
\newblock {\em Adv. Math.}, 225(6):3291--3326, 2010.

\bibitem[Sal01]{MR1851264}
Paolo Salvatore.
\newblock Configuration spaces with summable labels.
\newblock In {\em Cohomological methods in homotopy theory ({B}ellaterra,
  1998)}, volume 196 of {\em Progr. Math.}, pages 375--395. Birkh\"{a}user,
  Basel, 2001.

\bibitem[Seg73]{segal1973configuration}
Graeme Segal.
\newblock Configuration-spaces and iterated loop-spaces.
\newblock {\em Inventiones Mathematicae}, 21(3):213--221, 1973.

\bibitem[Seg78]{segal1978classifying}
Graeme Segal.
\newblock Classifying spaces related to foliations.
\newblock {\em Topology}, 17(4):367--382, 1978.

\bibitem[Ser79]{sergeraert1979bgamma}
Francis Sergeraert.
\newblock {B$\Gamma$} [d'apr{\`e}s {M}ather et {T}hurston].
\newblock In {\em S{\'e}minaire Bourbaki vol. 1977/78 Expos{\'e}s 507--524},
  pages 300--315. Springer, 1979.

\bibitem[Thu74]{thurston1974foliations}
William Thurston.
\newblock Foliations and groups of diffeomorphisms.
\newblock {\em Bulletin of the American Mathematical Society}, 80(2):304--307,
  1974.

\bibitem[Tsu06]{MR2284795}
Takashi Tsuboi.
\newblock On the group of foliation preserving diffeomorphisms.
\newblock In {\em Foliations 2005}, pages 411--430. World Sci. Publ.,
  Hackensack, NJ, 2006.

\bibitem[Tsu08]{MR2509723}
Takashi Tsuboi.
\newblock On the simplicity of the group of contactomorphisms.
\newblock In {\em Groups of diffeomorphisms}, volume~52 of {\em Adv. Stud. Pure
  Math.}, pages 491--504. Math. Soc. Japan, Tokyo, 2008.

\bibitem[Vas92]{MR1168473}
V.~A. Vassiliev.
\newblock {\em Complements of discriminants of smooth maps: topology and
  applications}, volume~98 of {\em Translations of Mathematical Monographs}.
\newblock American Mathematical Society, Providence, RI, 1992.
\newblock Translated from the Russian by B. Goldfarb.

\bibitem[Wei05]{weiss2005does}
Michael Weiss.
\newblock What does the classifying space of a category classify?
\newblock {\em Homology Homotopy Appl}, 7(1):185--195, 2005.

\end{thebibliography}
\end{document}